\documentclass[11pt]{article}

\usepackage[table, dvipsnames]{xcolor}
\usepackage[pagebackref,backref=true,colorlinks=true, linkcolor=Sepia, citecolor=Sepia]{hyperref}
\usepackage{amsmath,amsthm,amssymb,fancyhdr,graphicx,bbm,cancel,mathrsfs,todonotes,xypic,pinlabel}
\usepackage{tikz}
\usepackage{extarrows,tipa}
\usetikzlibrary{matrix}
\usepackage[all]{xy}
\usepackage{mathtools}
\usepackage{bm}
\usepackage{verbatim}
\usepackage{microtype}
\usepackage{subcaption}

\usepackage{quiver}

\theoremstyle{theorem}%The commands below will have bold title, italic text
\newtheorem{thm}{Theorem}[section]

\newtheorem{lem}[thm]{Lemma}
\newtheorem{prop}[thm]{Proposition}

\newtheorem{cor}[thm]{Corollary}
\theoremstyle{definition} %The commands below have bold title, standard text

\newtheorem{qu}[thm]{Question}

\newtheorem{rmk}[thm]{Remark}
\theoremstyle{remark} %The commands below have italic title, standard text
\usepackage{enumerate,pdfpages,stmaryrd}
\usepackage[margin=1.4in, marginparwidth=1in]{geometry}
\usepackage{tikz}
\usepackage{dsfont}
\usetikzlibrary{matrix}

\newcommand{\Z}{\mathbb{Z}}

\newcommand{\R}{\mathbb{R}}

\newcommand{\beq}{\begin{equation*}}
\newcommand{\eeq}{\end{equation*}}

\newcommand{\GL}{\mathrm{GL}}

\DeclareMathOperator{\Diff}{Diff}

\DeclareMathOperator{\Aut}{Aut}

\DeclareMathOperator{\id}{id}

\DeclareMathOperator{\Isom}{Isom}

\DeclareMathOperator{\SO}{SO}

\DeclareMathOperator{\Out}{Out}
\DeclareMathOperator{\Int}{int}

\DeclareMathOperator{\Flat}{Flat}

\newcommand{\dmo}{\DeclareMathOperator}

\newcommand{\Om}{\Omega}

\newcommand{\wtil}{\widetilde}

\newcommand{\bb}[1]{\mathbb{#1}}\newcommand{\mf}{\mathfrak}

\dmo{\sgn}{sign}\dmo{\Span}{span}
\dmo{\we}{\wedge}
\dmo{\ind}{ind}\dmo{\Ind}{Ind}
\dmo{\bop}{\bigoplus}\dmo{\pic}{Pic}
\dmo{\vol}{Vol}\dmo{\gal}{Gal}\dmo{\perm}{Perm}
\dmo{\tor}{Tor}\dmo{\ext}{Ext}\dmo{\Ext}{Ext}
\dmo{\aut}{aut}

\dmo{\inn}{Inn}\dmo{\var}{Var}
\dmo{\ad}{ad}\dmo{\curl}{curl}
\dmo{\hy}{\bb H}\dmo{\Sl}{SL}
\dmo{\psl}{PSL}
\dmo{\iso}{iso}
\dmo{\conf}{Conf}
\dmo{\stab}{Stab}\dmo{\Jac}{Jac }
\dmo{\diam}{diam}\dmo{\fix}{Fixed}\dmo{\Fix}{Fix}
\dmo{\injR}{injRad}\dmo{\Ad}{Ad}
\dmo{\esv}{ess-vol}
\dmo{\nil}{Nil}\dmo{\sol}{Sol}
\dmo{\Div}{div}
\dmo{\SU}{SU}
\dmo{\rk}{rk}
\dmo{\rank}{rank}
\dmo{\psp}{PSp}\dmo{\psu}{PSU}
\dmo{\PU}{PU}\dmo{\pgl}{PGL}
\dmo{\Mod}{Mod}\dmo{\range}{Range}
\dmo{\eu}{eu}\dmo{\mi}{mi}
\dmo{\Log}{Log}\dmo{\supp}{supp}
\dmo{\maps}{Maps}\dmo{\Gr}{Gr}
\dmo{\Pin}{Pin}
\dmo{\Spin}{Spin}\dmo{\Str}{Str}
\dmo{\Sq}{Sq}\dmo{\Symp}{Symp}
\dmo{\pd}{PD}\dmo{\PD}{PD}\dmo{\sig}{Sig}
\dmo{\St}{St}
\dmo{\Pt}{Pt}\dmo{\pt}{pt}

\dmo{\Pl}{PL}
\dmo{\String}{String}\dmo{\smear}{smear}

\dmo{\dev}{dev}
\dmo{\met}{Met}\dmo{\contact}{Contact}
\dmo{\teich}{Teich}\dmo{\Teich}{Teich}\dmo{\qi}{QI}
\dmo{\der}{Der}
\dmo{\cl}{Cliff}\dmo{\Cl}{Cl}
\dmo{\Pf}{Pf}
\dmo{\ch}{ch}\dmo{\diag}{diag}
\dmo{\grad}{grad}\dmo{\Char}{char}
\dmo{\spec}{Spec}\dmo{\Arg}{Arg}
\dmo{\gl}{GL}
\dmo{\sym}{Sym}\dmo{\Sym}{Sym}
\dmo{\com}{Comm}
\dmo{\Lk}{Lk}
\dmo{\CAT}{CAT}
\dmo{\Rep}{Rep}
\dmo{\Res}{Res}
\dmo{\Conf}{Conf}
\dmo{\PConf}{PConf}
\dmo{\Push}{Push}
\dmo{\Cont}{Cont}
\dmo{\sm}{\setminus}
\dmo{\vn}{\varnothing}
\dmo{\disk}{\mathbb D}
\dmo{\Trd}{Trd}\dmo{\Mat}{Mat}
\dmo{\Riem}{Riem}
\dmo{\Diffn}{\Diff_0}\dmo{\diff}{diff}
\dmo{\homeo}{Homeo}

\dmo{\Ham}{Ham}\dmo{\Met}{Met}
\dmo{\Ein}{Ein}\dmo{\CP}{\co P}
\dmo{\Per}{Per}\dmo{\Ric}{Ric}
\dmo{\Nrd}{Nrd}
\dmo{\Comp}{Comp}\dmo{\PSC}{PSC}
\dmo{\Cent}{Cent}\dmo{\Orb}{Orb}
\dmo{\aind}{a-ind}\dmo{\tind}{t-ind}
\dmo{\constant}{constant}
\dmo{\Td}{Td}
\dmo{\LMod}{LMod}
\dmo{\SMod}{SMod}
\dmo{\SDiff}{SDiff}
\dmo{\Br}{Br}
\dmo{\csch}{csch}
\dmo{\triv}{triv}
\dmo{\genus}{genus}
\dmo{\Homeq}{HomEq}
\dmo{\PP}{\mathbb{P}}
\dmo{\U}{U}
\dmo{\Gal}{Gal}
\dmo{\BDiff}{\wtil{\Diff}}
\dmo{\BAut}{\wtil{\Aut}}
\dmo{\Iso}{Iso}
\dmo{\Cone}{Cone}
\dmo{\codim}{codim}
\dmo{\II}{II}
\dmo{\I}{I}
\dmo{\InjRad}{InjRad}
\dmo{\Inn}{Inn}
\dmo{\sys}{sys}
\dmo{\Comm}{Comm}
\dmo{\PO}{PO}
\dmo{\vertex}{Vert}
\dmo{\POm}{P\Om}
\dmo{\ab}{ab}
\dmo{\PSO}{PSO}
\dmo{\CRS}{CRS}
\dmo{\Diffext}{Diffext}
\dmo{\Diffextad}{Diffextad}
\dmo{\Diffstand}{Diffstand}

\newcommand{\EE}{\mathbb{E}}

\newcommand{\HH}{\mathbb{H}}

\newcommand{\RR}{\mathbb{R}}

\newcommand{\ZZ}{\mathbb{Z}}

\newcommand{\cM}{\mathcal{M}}

\newcommand{\cO}{\mathcal{O}}

\newcommand{\cT}{\mathcal{T}}

\newcommand{\cX}{\mathcal{X}}
\newcommand{\cY}{\mathcal{Y}}
\newcommand{\cZ}{\mathcal{Z}}

\newcommand{\kb}{\mathfrak{b}}

\newcommand{\Id}{\operatorname{Id}}

\renewcommand{\II}{\operatorname{I\!I}}

\usepackage[titles]{tocloft}

\begin{document}

\title{Finite group actions on nonpositively curved 3-manifolds}
\author{Bena Tshishiku and Zhengyu Zou}

\maketitle

%\vspace*{-4ex}

\begin{abstract}
We show that every orientation-preserving action of a finite group on a closed nonpositively curved 3-manifold preserves a nonpositively curved metric.
\end{abstract}

\setcounter{secnumdepth}{3}
\setcounter{tocdepth}{1}
\begingroup
\linespread{0.8}\selectfont % Forces lines inside the TOC to pack tightly together
\tableofcontents
\addtocontents{toc}{\protect\thispagestyle{empty}}
\endgroup

\section{Introduction}

This paper concerns finite group actions on closed, orientable 3-manifolds $M$ with metrics of nonpositive sectional curvature. We are motivated by the following question of Schoen--Yau \cite[pg.\ 378]{schoen-yau}.

\begin{qu}[Schoen--Yau]\label{q:main}
Let $M$ be a closed manifold with nonpositive curvature, and let $G$ be a finite subgroup of outer automorphisms of $\pi_1(M)$. Is there a nonpositively curved metric $\mu$ on $M$ such that $G$ is induced by a finite group of isometries of $(M,\mu)$?
\end{qu}

For surfaces, Question \ref{q:main} was asked earlier by Nielsen \cite[\S12]{nielsen} and was answered affirmatively by Kerckhoff \cite{kerckhoff}. In high dimensions, Farrell--Jones \cite{FJ-nielsen} provided a negative answer to Question \ref{q:main} by finding negatively curved $M$ and finite order outer automorphisms that are not induced by \emph{any} diffeomorphisms of $M$ (finite order or not!).

In dimension 3, any finite $G<\Out(\pi_1M)$ is induced by a finite group of \emph{diffeomorphisms} of $M$ by work of Zimmermann and Zieschang--Zimmermann \cite{zimmermann,zimmermann-zieschang, zimmermann-extensions}, so Question \ref{q:main} reduces to the question of whether any finite group action by diffeomorphisms on a nonpositively curved $M$ preserves some nonpositively curved metric. We answer this affirmatively for orientation-preserving actions.

\begin{thm}[Invariant nonpositively curved metrics]\label{thm:main}
Let $M$ be a closed, orientable $3$-manifold, and let $G$ be a finite group of orientation-preserving diffeomorphisms of $M$. If $M$ has a nonpositively curved metric, then there exists a nonpositively curved metric that's preserved by $G$.
\end{thm}

Theorem \ref{thm:main} was previously known when either (1) $M$ has trivial JSJ decomposition (i.e.\ $M$ is either Seifert fibered or atoroidal) or (2) $G$ acts freely on $M$; see Section \ref{sec:previous-work}. In terms of what was not previously known, the most interesting case of Theorem \ref{thm:main} is when $M$ is a graph manifold and the action of $G$ on $M$ is not free; see Section \ref{sec:proof-discussion} for some discussion of the proof. 

Combining Theorem \ref{thm:main} with the work of Zimmermann and Zieschang--Zimmermann mentioned above, we answer Question \ref{q:main} in dimension 3 for $G$ in the orientation-preserving subgroup $\Out^+(\pi_1M)$, a subgroup of $\Out(\pi_1M)$ of index at most 2. 

\begin{cor}[Geometric Nielsen realization]\label{cor:main}
Let $M$ be a closed, oriented, nonpositively curved $3$-manifold. For any finite subgroup $G<\Out^+(\pi_1M)$, there is a nonpositively curved metric $\mu$ so that $G$ is induced by a finite subgroup $H$ of isometries of $(M,\mu)$.
\end{cor}

We note that the finite subgroup $H<\Isom(M,\mu)$ in the statement of Corollary \ref{cor:main} is not necessarily isomorphic to the given group $G$. Corollary \ref{cor:main} only gives a finite group $H<\Isom(M,\mu)$ such that the image of $H\hookrightarrow\Isom(M,\mu)\to \Out(\pi_1M)$ is $G$. The map $H\to G$ may have (finite) kernel. This extra flexibility is necessary only when $M$ is Seifert fibered because in that case not every subgroup $G<\Out(\pi_1M)$ can be lifted to $\Diff(M)$; nevertheless, Zimmermann \cite{zimmermann-extensions} showed that there is always a finite extension of $G$ (like $H$) that does lift.

\subsection{Previous work related to Theorem \ref{thm:main}} \label{sec:previous-work}

There are two main cases where Theorem \ref{thm:main} was previously known. 
 
If the JSJ decomposition of $M$ is trivial, then $M$ is geometric (in the sense of Thurston geometries) by geometrization; see \cite{scott,perelman-entropy,perelman-extinction,perelman-surgery}. For geometric $M$, every finite group action on $M$ preserves a geometric metric by work of Meeks--Scott \cite[Thm.\ 2.1]{meeks-scott} and Dinkelbach--Leeb \cite[Thm.\ H]{dinkelbach-leeb}. This stronger statement implies Theorem \ref{thm:main} when $M$ has trivial JSJ decomposition. Note that a geometric 3-manifold that has a nonpositively curved metric has one of the Thurston geometries $\mathbb H^3$, $\mathbb E^3$, or $\mathbb H^2\times\R$ \cite{eberlein}. 

If the JSJ decomposition of $M$ nontrivial and $G$ acts \emph{freely} on $M$, then Theorem \ref{thm:main} was shown by the second-named author in \cite{byron}, using in particular work of Leeb \cite{leeb}. The assumption that $G$ acts freely on $M$ implies that the quotient $N:=M/G$ is a manifold, and it suffices to show that $N$ has a nonpositively curved metric. If the JSJ decomposition of $N$ has at least one atoroidal component, then $N$ is nonpositively curved by Leeb \cite{leeb}. If the JSJ decomposition of $N$ has only Seifert fibered components (i.e. $N$ is a graph manifold), the situation is subtle because not all graph manifolds admit nonpositively curved metrics. It is a priori conceivable that a nonpositively curved graph manifold could cover a graph manifold with no nonpositively curved metric, but \cite{byron} shows that this does not occur.

Given these results, to prove Theorem \ref{thm:main}, it remains to consider the case when the JSJ decomposition of $M$ is \emph{nontrivial} and $G$ does \emph{not} act freely on $M$. 
In this case $M/G$ is a 3-dimensional orbifold. We extend the arguments from \cite{byron} to this setting, overcoming several technical issues caused by orbifold singularities. While our proof builds on the ideas of \cite{byron}, this paper can be read independently.

\subsection{About the proof of Theorem \ref{thm:main}}\label{sec:proof-discussion}

Let $M$ be a closed, orientable 3-manifold and let $G$ be a finite group acting faithfully on $M$ by orientation-preserving diffeomorphisms. The quotient $M/G$ is a 3-dimensional orbifold that we denote by $\mathcal M$. Our goal is to show that $\mathcal M$ has a nonpositively curved metric. (Since $\mathcal M$ is only an orbifold, by definition this means that the universal cover has a nonpositively curved metric for which the deck group $\pi_1(\mathcal M)$ acts isometrically, where $\pi_1(\cM)$ denotes the orbifold fundamental group.) 

We use the JSJ decomposition to organize the problem. As for irreducible 3-manifolds, the orbifold $\mathcal M$ has a JSJ decomposition: there exists a unique (up to isotopy) sub-orbifold $\mathcal Z\subset \mathcal M$ whose components are incompressible 2-dimensional orbifolds with zero Euler characteristic such that complement $\cM\setminus\cZ$ consists of atoroidal or Seifert fibered components; see e.g.\ \cite[Thm.\ 3.3]{BMP}.

\paragraph{Geometric extension problem for flat boundary-metrics.} 

Ultimately we will show that $\mathcal M$ has a nonpositively curved metric with the additional property that a JSJ orbifold $\mathcal Z\subset\mathcal M$ is totally geodesic and flat. In particular, our strategy is to find a flat metric on $\cZ$ that extends to a nonpositively curved metric on each component of $\cM\mid\cZ$, where $\cM\mid\cZ=(\cM\setminus\cZ)\cup(\cZ\sqcup\cZ)$ denotes the real blow up of $\cM$ along $\cZ$. 

Toward this goal, it is helpful to know, for an atoroidal or Seifert fibered orbifold $\mathcal X$, when a flat metric on $\partial\mathcal X$ extends to a nonpositively curved metric on $\mathcal X$ for which the boundary is totally geodesic. When such a nonpositively curved metric extension exists we say that the flat boundary-metric is \emph{geometrically extendable}. We study this problem thoroughly in Section \ref{sec:extend}. 

When $\mathcal X$ is atoroidal, every flat boundary-metric is geometrically extendable (Proposition \ref{prop:atoroidal}); this generalizes a result of Leeb \cite[Prop.\ 2.3]{leeb} from manifolds to orbifolds. 

The Seifert orbifold case is more subtle. When $X$ is a Seifert fibered \emph{manifold}, Leeb gives a simple characterization of when a flat boundary-metric is geometrically extendable. We extend Leeb's result to certain Seifert orbifolds and also recast it (co)homologically: 
\begin{itemize}
\item (homological criterion for extending geometrically) Let $\cX\to\cO$ be an orientable Seifert fibered orbifold. Assume that $\cO$ orientable and $\cX\neq T^2\times[0,1]$. Given a flat metric $\mu_\partial$ on $\partial\cX$, there exists a class $\beta_{\mu_\partial}\in H^1(\partial\cX;\R)$ with the property that $\mu_\partial$ extends geometrically if and only if $\beta_{\mu_\partial}$ is in the image of the homomorphism $H^1(\cX;\R)\to H^1(\partial\cX;\R)$ induced by inclusion. 
\end{itemize}

See Proposition \ref{prop:leeb} for a statement that also includes the case $\cX=T^2\times[0,1]$. 

When the base orbifold $\cO$ is orientable (as in the statement of the preceding homological criterion), we call $\cX\to\cO$ \emph{totally orientable}. Any orientable $\mathcal X$ has a canonical totally-orientable double cover $\mathcal X^+\to\mathcal X$ that covers  the orientation double cover of the base orbifold. Using Proposition \ref{prop:leeb}, we show the following. 
\begin{itemize}
\item A flat boundary-metric on $\mathcal X$ extends geometrically  if and only if its pullback to $\cX^+$ extends geometrically  (Proposition \ref{prop:totally-orientable-cover}). 
\item If $X\to\cX^+$ is a cover, then a flat boundary-metric on $\cX^+$ extends geometrically  if and only if its pullback to $X$ extends geometrically  (Proposition \ref{prop:cover-totally-orientable}).  
\end{itemize}

\paragraph{Assembling nonpositively curved metrics.}

To construct a nonpositively curved metric on $\mathcal M=M/G$, we separate two cases depending on whether or not $M$ (or equivalently $\mathcal M$) has an atoroidal component in its JSJ decomposition. The case when there is an atoroidal component is easier and is similar to the manifold case in \cite{leeb}; see Section \ref{sec:atoroidal-proof}. Here we focus on the more interesting case when $M$ is a graph manifold. We consider the following pullback diagram. 
\begin{equation} \label{eqn:regular-char-cover-diagram}
    \begin{xy}
    (-20,0)*+{N}="A";
    (0,0)*+{\mathcal M^+}="B";
    (0,-12)*+{\mathcal M}="D";
    (-20,-12)*+{M}="C";
    {\ar"A";"C"}?*!/_3mm/{};
    {\ar "B";"D"}?*!/_3mm/{};
    {\ar "C";"D"}?*!/^2mm/{\scriptstyle G};
    {\ar "A";"B"}?*!/_3mm/{};
    {\ar "A";"D"}?*!/_2mm/{\scriptstyle H};
    \end{xy}
\end{equation}
Here $\mathcal M^+\to\cM$ is a double cover such that each component of the JSJ decomposition of $\cM^+$ is totally orientable. Using this diagram, let $\mathcal Z\subset\mathcal M$ be a JSJ orbifold, and let $Z\subset N$ be the pre-image of $\mathcal Z$ under $N\to\mathcal M$. Then $Z$ gives a decomposition of $N$ into atoroidal and Seifert components, but it might not be minimal (some components of $N\mid Z$ may be $T^2\times[0,1]$). By assumption $M$ has a nonpositively curved metric, so $N$ is also nonpositively curved. Writing $H$ for the deck group of $N\to\cM$, we show that there is a flat metric on $Z$ that is both $H$-invariant and geometrically extendable on each component of $N\mid Z$ (Proposition \ref{prop:average-JSJ}); a key ingredient in this argument is the observation that for any two flat metrics $\mu_1,\mu_2$ on either $T^2$ or $S^2(2,2,2,2)$, there is a third flat metric that has the same Levi-Civita connection as $\mu_1$ and projects to the same point in the Teichm\"uller space as $\mu_2$ (see Lemma \ref{lem:torus-pillow}). By construction, the flat metric on $Z$ satisfies the homological criterion (Proposition \ref{prop:leeb}) on each component of $N\mid Z$. Since the metric on $Z$ is $H$-invariant, it descends to a metric on $\cZ$, and this metric extends geometrically  on each component of $\cM\mid\cZ$ by the aforementioned Propositions \ref{prop:totally-orientable-cover} and \ref{prop:cover-totally-orientable}. 

\paragraph{Conventions.} Throughout this paper, curvature always means sectional curvature. We work in the smooth category of manifolds and orbifolds. All 3-dimensional orbifolds are assumed to be orientable. 

\paragraph{Acknowledgements.}The first author thanks Shmuel Weinberger for pointing out that Question \ref{q:main} was unresolved in dimension 3. 

\paragraph{Computational resource disclosure.}
The authors have used LLMs to find references, check the proofs, and proofread the paper. The authors take full responsibility for the mathematical content of this paper.

\section{Flat metrics on Euclidean 2-orbifolds} \label{sec:2-orbifolds}

By a Euclidean 2-orbifold we mean a 2-orbifold that admits a flat metric. These are the topological 2-orbifolds with zero Euler characteristic. Up to (orbifold) diffeomorphism, there are five closed orientable Euclidean 2-orbifolds, pictured in Figure \ref{fig:surface-orbifolds}.  We refer to $S^2(2,2,2,2)$ as the pillow and each of $S^2(3,3,3)$, $S^2(2,4,4)$, and $S^2(2,3,6)$ as a turnover. These Euclidean 2-orbifolds appear in the JSJ decomposition of a 3-orbifold. Our procedure for putting nonpositively curved metrics on 3-orbifolds is guided by the behavior of flat metrics on Euclidean 2-orbifolds. In this section we collect some facts about these flat metrics that will be used in later sections.

\begin{figure}[h!]
\centering
% \begin{tikzpicture}[x=1pt,y=1pt]
% \node[anchor=south west,inner sep=0] at (0,0)
%     {\includegraphics[scale=.4]{2-orbifolds.pdf}};
% \node at (75,85) {$T^2$};
% \node at (180,85) {$S^2(2,2,2,2)$};
% \node at (45,-5) {$S^2(3,3,3)$};
% \node at (140,-5) {$S^2(2,4,4)$};
% \node at (245,-5) {$S^2(2,3,6)$};
% \end{tikzpicture}

    \def\svgwidth{0.7\columnwidth}
    %% Creator: Inkscape 1.3.2 (091e20e, 2023-11-25), www.inkscape.org
%% Accompanies image file 'flat-2-orbifolds.pdf' (pdf, eps, ps)
\begingroup%
  \makeatletter%
  \providecommand\color[2][]{%
    \errmessage{(Inkscape) Color is used for the text in Inkscape, but the package 'color.sty' is not loaded}%
    \renewcommand\color[2][]{}%
  }%
  \providecommand\transparent[1]{%
    \errmessage{(Inkscape) Transparency is used (non-zero) for the text in Inkscape, but the package 'transparent.sty' is not loaded}%
    \renewcommand\transparent[1]{}%
  }%
  \providecommand\rotatebox[2]{#2}%
  \newcommand*\fsize{\dimexpr\f@size pt\relax}%
  \newcommand*\lineheight[1]{\fontsize{\fsize}{#1\fsize}\selectfont}%
  \ifx\svgwidth\undefined%
    \setlength{\unitlength}{622.79980277bp}%
    \ifx\svgscale\undefined%
      \relax%
    \else%
      \setlength{\unitlength}{\unitlength * \real{\svgscale}}%
    \fi%
  \else%
    \setlength{\unitlength}{\svgwidth}%
  \fi%
  \global\let\svgwidth\undefined%
  \global\let\svgscale\undefined%
  \makeatother%
  \begin{picture}(1,0.59891124)%
    \lineheight{1}%
    \setlength\tabcolsep{0pt}%
    \put(0,0){\includegraphics[width=\unitlength,page=1]{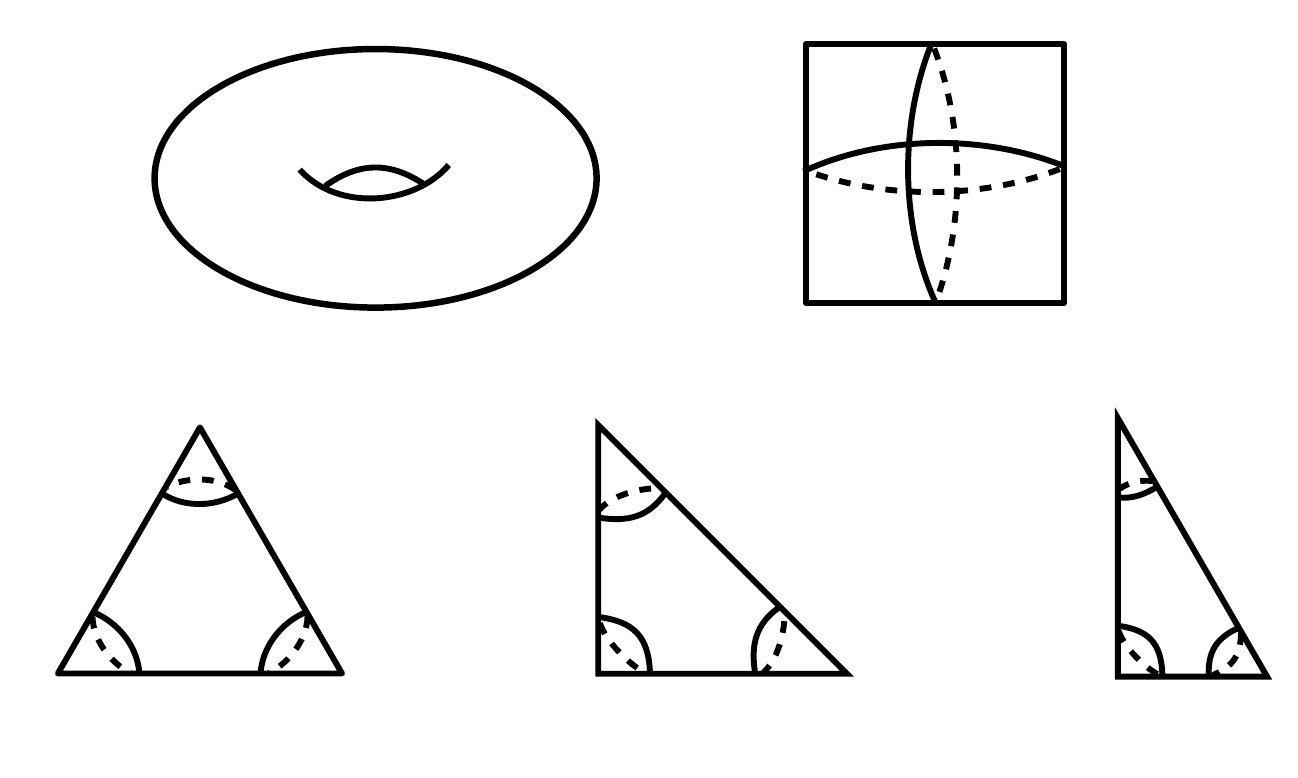}}%
    \put(0.27234226,0.30611133){\color[rgb]{0,0,0}\makebox(0,0)[lt]{\lineheight{1.25}\smash{\begin{tabular}[t]{l}$T^2$\end{tabular}}}}%
    \put(0.65276601,0.30726795){\color[rgb]{0,0,0}\makebox(0,0)[lt]{\lineheight{1.25}\smash{\begin{tabular}[t]{l}$S^2(2, 2, 2, 2)$\end{tabular}}}}%
    \put(0.10458122,0.02220852){\color[rgb]{0,0,0}\makebox(0,0)[lt]{\lineheight{1.25}\smash{\begin{tabular}[t]{l}$S^2(3, 3, 3)$\end{tabular}}}}%
    \put(0.48502866,0.02332822){\color[rgb]{0,0,0}\makebox(0,0)[lt]{\lineheight{1.25}\smash{\begin{tabular}[t]{l}$S^2(2, 4, 4)$\end{tabular}}}}%
    \put(0.86299713,0.02410436){\color[rgb]{0,0,0}\makebox(0,0)[lt]{\lineheight{1.25}\smash{\begin{tabular}[t]{l}$S^2(2, 3, 6)$\end{tabular}}}}%
  \end{picture}%
\endgroup%

\caption{Closed, orientable Euclidean 2-orbifolds. Numbers in parentheses indicate the order of the cone points.}
\label{fig:surface-orbifolds}
\end{figure}

Two of the main results in the section (Lemmas \ref{lem:turnover} and \ref{lem:torus-pillow}) concern when two flat metrics on $\cO$

\subsection{Flat metrics up to isotopy on Euclidean 2-orbifolds} 

We say two metrics $\mu,\mu'$ on an orbifold $\cO$ are \emph{isotopic} if there exists a diffeomorphism isotopic to the identity $\phi\in\Diff_0(\cO)$ such that $\phi^*(\mu)=\mu'$. In this section we recall the classification of flat metrics up to isotopy. See also \cite{bettiol}. 

Regarding the definition of isotopic, we recall that the diffeomorphism group of a (good) orbifold $\cO$ is defined as the normalizer of the deck group $\pi_1(\cO)$ in the diffeomorphism group $\Diff(\widetilde{\cO})$ of the universal cover. We also remark that, by convention, if $\mathcal O$ is an orbifold, then $\pi_1(\cO)$ denotes the orbifold fundamental group. When we wish to refer to the fundamental group of the underlying topological space, we write $\pi_1(|\cO|)$.

\subsubsection{Flat metrics on Euclidean turnovers} \label{subsec:interpolate-turnovers}

\begin{lem} \label{lem:turnover}
Let $\mathcal{T}$ be a Euclidean turnover. Any two flat metrics on $\cT$ with the same area are isotopic. 
\end{lem}

Lemma \ref{lem:turnover} follows from two simple facts. First, the moduli space of flat metrics on $\cT$ with fixed area is a point, since $\cT$ is obtained by doubling a Euclidean triangle whose angles are determined by the cone orders, so the flat metric is unique up to scaling and diffeomorphism. Second, every mapping class of $\cT$ is represented by an isometry of any flat metric on $\cT$ because symmetries of the triangle induce isometries of the turnover; here it is helpful to recall that the mapping class group of a turnover is isomorphic to the group of order-preserving permutations of its cone points. 

Lemma \ref{lem:turnover} can be interpreted as saying the Teichm\"uller space $\mathrm{Teich}(\mathcal T)$ of flat metrics on a turnover is diffeomorphic to $\mathbb R_+=(0,\infty)$, given by the area. Indeed, after fixing the area, any two flat metrics with the same area differ by a diffeomorphism, and this diffeomorphism can be chosen up to isotopy to be an isometry.

\subsubsection{Flat Metrics on the torus and the pillow } \label{sec:torus-pillow}

In this section $\mathcal T$ denotes either the torus $T^2$ or the pillow $S^2(2,2,2,2)$. Our main observation in this section is the following Lemma \ref{lem:torus-pillow}. To state it, we say two flat metrics on $\mathcal T$ are \emph{affinely equivalent} if they have the same Levi-Civita connection. Recall (see e.g.\ \cite[Ch.\ 4,5]{lee}) that the Levi-Civita connection is the unique torsion-free affine connection $\nabla$ for which the metric is parallel $\nabla \mu=0$ (by the fundamental theorem of Riemannian geometry). An affine connection defines parallel transport (of tangent vectors along curves) and in particular it determines geodesics (curves whose tangent vector is parallel); affinely equivalent metrics have the same geodesics but differ according to the length of these geodesics and the angles at which they meet.

\begin{lem}\label{lem:torus-pillow}
Let $\mathcal T$ be the torus or pillow. For any two flat metrics on $\cT$, there is a unique third metric that is affinely equivalent to the first and isotopic to the second. 
\end{lem}

Lemma \ref{lem:torus-pillow} plays a key role in the proof of Theorem \ref{thm:main}. It is used in Proposition \ref{prop:atoroidal}, Lemma \ref{lem:free-boundary}, and Proposition \ref{prop:average-JSJ}. 

\begin{proof}[Proof of Lemma \ref{lem:torus-pillow}]
Let $\Flat(\mathcal T)$ denote the space of flat metrics on $\mathcal T$. 
Fixing a basepoint $*\in\mathcal T$; if $\cT$ is the pillow we choose $*$ to be one of the cone points. There is a fiber sequence
\[
\Diff_0(\mathcal T,*)\to\Flat(\mathcal T)\to \Flat(\mathcal T)/\Diff_0(\mathcal T,*). 
\] 
See \cite[Thm.\ 7.4]{ebin}, which is generalized to orbifolds in \cite[Thm.\ 3.1.2]{lund}. For the torus, the basepoint is needed for this fibration to be a principal bundle, since each flat metric has a translation group of isometries. For the pillow, with our choice of $*$, we have $\Diff_0(\cT,*)=\Diff_0(\cT)$. In either case the quotients $\mathrm{Flat}(T^2)/\Diff_0(T^2)$ and $\mathrm{Flat}(T^2)/\Diff_0(T^2,*)$ are homeomorphic.

In brief, we show that for each $\mu\in\Flat(\mathcal T)$ the set $A(\mu)$ of flat metrics affinely equivalent to $\mu$ defines a section of this bundle. Since the bundle is a principal bundle this directly implies the claim. 

\paragraph{Claim 1.} There is a homeomorphism $\Flat(\mathcal T)/\Diff_0(\mathcal T)\cong \GL_2^+(\R)/\SO(2)$. This is well-known for our cases of interest; Teichm\"uller spaces of flat metrics on Euclidean orbifolds are studied more generally in \cite[\S5]{bettiol}. We give a short explanation of Claim 1 using some basic differential geometry, which we will also use to give a natural section of this bundle through any $\mu\in\Flat(\mathcal T)$ (and use again in Proposition \ref{prop:leeb}). 

To define a map 
\[\begin{array}{rrcll}
B:&\Flat(\mathcal T)&\to&\GL_2^+(\R)/\SO(2)\\
&\mu&\mapsto&B_\mu\end{array}\] (c.f.\ \cite[\S2.3]{leeb} in the torus case) we interpret the codomain as the space of inner products on $\R^2$ (a coset $x\SO(2)$ determines the inner product with matrix $xx^t$). Let $T^2\to\mathcal T$ be the cover corresponding to the maximal abelian subgroup of $\pi_1(\mathcal T)$. Let $b,f\in H_1(T^2)$ be a basis. Given a flat metric $\mu$, let $b_\mu,f_\mu$ be geodesic representatives for $b,f$ with respect to $\mu$ (or technically its pullback to $T^2$) and define an inner product $B_\mu$ on $H_1(T^2)\otimes\R$ by 
\begin{equation}\label{eqn:inner-product}B_\mu(b,b)=\mathrm{Length}(b_\mu)^2, \>\>\>B_\mu(f,f)=\mathrm{Length}(f_\mu)^2, \end{equation}
and $B_\mu(b,f)$ is the $\mu$-inner product between the tangent vectors to $b_\mu$ and $f_\mu$ measured in the tangent space of the unique point of intersection. 

It is easy to see that $B$ is surjective and factors through $\Flat(\mathcal T)\to\Flat(\mathcal T)/\Diff_0(\mathcal T)$. We show that the induced map $\Flat(\mathcal T)/\Diff_0(\mathcal T)\to\GL_2^+(\R)/\SO(2)$ is injective. 

First consider the case when $\mathcal T=T^2$. Fix a basepoint $*\in T^2$. Let $\mu,\nu$ be flat metrics and assume $B_\mu=B_\nu$. We will show there is $\phi\in\Diff_0(T^2,*)$ so that $\phi^*\mu=\nu$. Fix a basis $b,f\in H_1(T^2)$ and let $V^b_\mu,V^f_\mu$ be $\mu$-parallel unit vector fields whose trajectories are closed geodesics representing $b,f$, respectively. Define $V^b_\nu,V^f_\nu$ similarly. By assumption the orbits of $V^b_\mu$ and $V^b_\nu$ (resp.\ $V^f_\mu$ and $V^f_\nu$) have the same length, which we denote $\ell_b$  (resp.\ $\ell_f$). Define $\phi_\mu: T^2\to S^1\times S^1$ by sending the point 
\[\mathrm{Flow}(V^b_\mu,t_1)\circ \mathrm{Flow}(V^f_\mu,t_2)(*)\] in $T^2$ to $(t_1,t_2)\in S^1\times S^1=\R/\ell_b\Z\times\R/\ell_f\Z$. Define $\phi_\nu$ similarly. The composition $\phi:=\phi_\nu^{-1}\phi_\mu:T^2\to T^2$ is a diffeomorphism fixing $*$. By construction $(\phi_\mu)_*(\mu)=(\phi_\nu)_*(\nu)$ (because $B_\mu=B_\nu$), so $\phi_*(\mu)=\nu$. Furthermore, $\phi_*(V^b_\mu)=V^b_\nu$ and $\phi_*(V^f_\mu)=V^f_\nu$ implies $\phi$ acts trivially on $H_1(T^2)$, so $\phi$ is isotopic to the identity.

We proceed similarly when $\mathcal T$ is the pillow. Given metrics $\mu,\nu$ on $\mathcal T$, we obtain metrics $\mu',\nu'$ on $T^2$ via pullback along the hyperelliptic double cover $T^2\to\mathcal T$. We apply the previous argument -- with basepoint $*\in T^2$ a fixed point of the hyperelliptic involution $\tau$ -- to get $\phi'\in\Diff_0(T^2)$ such that $(\phi')_*(\mu')=\nu'$. By the construction, $\phi'$ commutes with the hyperelliptic involution (for the coordinate charts $\phi_{\mu'},\phi_{\nu'}:T^2\to S^1\times S^1$, we have $\phi_{\mu'}\tau\phi_{\mu'}^{-1}=\phi_{\nu'}\tau\phi_{\nu'}^{-1}$), so $\phi'$ descends to $\phi\in\Diff(\mathcal T)$ such that $\phi_*(\mu)=\nu$. By construction $\phi$ preserves the isotopy class of simple closed curves $\bar b,\bar f$ on $\mathcal T$ that lift to $b,f$, and $\phi$ does not permute the cone points (because $\phi'$ acts trivially on $H_1(T^2;\Z/2\Z)$, which can be identified with the 2-torsion points, the fixed points of $\tau$) and this implies that $\phi$ is isotopic to the identity by the Alexander method \cite[Prop 2.6, Prop 2.7]{farb-margalit}.

\paragraph{Claim 2.} $A(\mu)$ is the image of a section of $B:\Flat(\mathcal T)\to \GL_2^+(\R)/\SO(2)$. 

As above, consider the cover $T^2\to\mathcal T$ corresponding to the maximal abelian subgroup of $\pi_1(\mathcal T)$. Fix a basis $b,f$ for $H_1(T^2)$ and a basepoint $*\in T^2$. 

Write $\nabla$ for the Levi-Civita connection of $\mu'$, the pullback metric on the torus. Since the affine connection $\nabla$ is flat and has trivial linear holonomy, a $\nabla$-parallel metric $\nu$ is determined by its value at the basepoint $*$. Furthermore, the connection determines the geodesics, so if $b_\nabla,f_\nabla:[0,1]\to T^2$ are the geodesic representatives of $b,f$ based at $*$, then we have 
\[
\begin{array}{lcl}
B_\nu(f,f)&=&\mathrm{Length}_\nu(f_\nabla)^2=\nu(\dot{f_\nabla}(0),\dot{f_\nabla}(0))\\
B_\nu(b,b)&=&\mathrm{Length}_\nu(b_\nabla)^2=\nu(\dot{b_\nabla}(0),\dot{b_\nabla}(0))\\
B_\nu(b,f)&=&\nu(\dot{b_\nabla}(0),\dot{f_\nabla}(0))\\
\end{array}\]
This shows that $B_\nu$ is determined by the value of $\nu$ at the basepoint. Thus the restriction of $B$ to $A(\mu)$ is a bijection, which shows that $A(\mu)$ is the image of a section. This completes the proof of Claim 2 and the proof of the lemma. 
\end{proof}

\begin{rmk}[Averaging affinely equivalent metrics]\label{rmk:average}
A key property of affinely equivalent metrics is that they form a convex cone in the space of metrics (this is a completely general fact). For flat metrics on the torus or pillow, we use additionally that convex combinations of affinely equivalent flat metrics remain flat (e.g.\ since the flat metric is determined by its value at a basepoint). Consequently, the average of affinely equivalent flat metrics is another flat metric. 
\end{rmk}

\begin{rmk}[Normal form]\label{rmk:normal-form}
Given two inner products on $\R^2$, up to change of coordinates, we can assume the first is the identity and the second is diagonal (by the spectral theorem). 
\end{rmk}

\subsection{Geometric metrics on 2-orbifolds with given boundary lengths}\label{sec:boundary-lengths}

The following Lemma \ref{lem:boundary-lengths} is used in Propositions \ref{prop:leeb} and \ref{prop:totally-orientable-cover}. These both concern the problem of extending flat boundary-metrics to nonpositively curved metrics on Seifert fibered orbifolds. 

In the statement $\bar I$ denotes a mirrored interval, which is the 1-dimensional orbifold quotient of $S^1$ by a reflection. For a 2-orbifold $\cO$, we say that an assignment of lengths to the boundary components of $\cO$ is \emph{geometrically extendable} if there exists a constant-curvature metric on $\cO$ with totally geodesic boundary whose boundary curves have the prescribed lengths. 

\begin{lem}[Realizing boundary lengths] \label{lem:boundary-lengths}
Let $\cO$ be a compact $2$-orbifold with $\chi(\cO) \leq 0$. If $\cO$ is not $S^1 \times [0,1]$ or $\bar I \times [0,1]$, then any assignment of lengths to the components of $\partial\cO$ is geometrically extendable. If $\cO \cong  S^1 \times [0,1]$ or $\bar I \times [0,1]$, then the same conclusion holds if we assume additionally that the two boundary lengths are equal. 
\end{lem}

\begin{proof}
Let $n$ be the number of boundary components of $\cO$. 

When $\chi(\cO)<0$, the result follows from the fact that the function $\mathcal T(\cO)\to(\R^+)^n$ on the Teichm\"uller space of  $\cO$ that records the boundary lengths is a fibration  \cite[\S7.4.4]{choi}. 

If $\chi(\cO)=0$, then $\cO$ is one of the seven $2$-orbifolds in Figure \ref{fig:annular-orbifolds}; cf.\ \cite[\S1.4]{choi-goldman}. Each of these orbifolds is finitely covered by $S^1\times[0,1]$ and, apart from $S^1\times[0,1]$ and $\bar I\times[0,1]$, each has only one boundary component. Then we can construct a flat structure on $\cO$ with any given boundary length by taking a quotient of a flat annulus with the appropriate boundary length. If $\cO$ is $S^1 \times [0,1]$ or $\bar I \times [0,1]$, then flat metrics are product metrics, so the two boundary components have the same length. 
\end{proof}

\begin{figure}[h!]
\centering

    \def\svgwidth{0.7\columnwidth}
    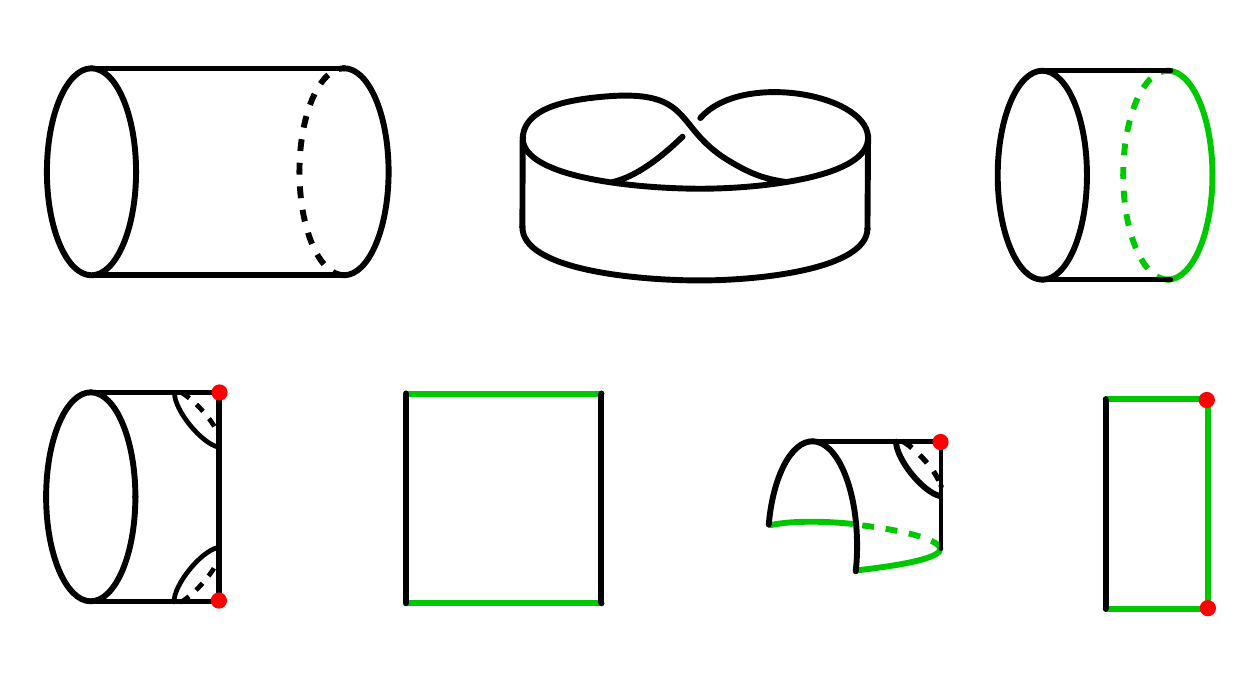

\caption{The seven compact  2-orbifolds  with $\partial\cO\neq\varnothing$ and $\chi(\cO)=0$. Green curves are reflector singularities, and red points are order-$2$ cone points or corner reflectors.}
\label{fig:annular-orbifolds}
\end{figure}

\section{Seifert fibered 3-orbifolds} \label{sec:seifert}

In this section we collect some facts about the structure of Seifert fibered orbifolds that will be used in later sections. 

Similar to a Seifert fibered manifold, a Seifert fibered orbifold $\cX$ has a decomposition into circles and mirrored intervals (called the fibers) such that each fiber has a saturated neighborhood of the form $(D^2\times S^1)/G$, where $G$ preserves each factor, acting orthogonally, and the fibers are the images of vertical fibers of $D^2\times S^1$; if $\cX$ has nonempty boundary, then on the boundary, the local model is $(D_+^2\times S^1)/G$, where $D_+^2$ is a half-disk; cf.\ \cite[pg.\ 201]{boileau-leeb-porti}. 

The decomposition of $\cX$ into fibers give a quotient $\cX\to\cO$, where $\cO$ is a 2-dimensional orbifold, which we call the base orbifold. 

\subsection{Pullbacks and totally-orientable covers}\label{sec:TO-cover}

First we observe that given an orientable Seifert orbifold $\pi:\cX\to\cO$ and a finite cover $q: \hat\cO \rightarrow \cO$ of the base orbifold, there exists a Seifert orbifold $\hat\pi: \hat\cX \rightarrow \hat\cO$ and a covering map $p: \hat\cX \rightarrow \cX$ such that $\pi\circ p=q\circ\hat\pi$. Furthermore, the preimage of every regular fiber $f$ of $\cX$ consists of $\deg(q)$ regular fibers, each mapped homeomorphically under $p$. Indeed, by \cite[Proposition 2.12]{BMP}, there is a short exact sequence 
\[1 \rightarrow C \rightarrow \pi_1(\cX) \rightarrow \pi_1(\cO) \rightarrow 1,\]
where $C$ is cyclic. The cover $q: \hat\cO \rightarrow \cO$ induces an inclusion of orbifold fundamental groups $\pi_1(\hat\cO) \hookrightarrow \pi_1(\cO)$. Define $\Gamma < \pi_1(\cX)$ as the pre-image of $\pi_1(\hat\cO)<\pi_1(\cO)$ under $\pi_1(\cX)\to\pi_1(\cO)$. 
%pullback in the following diagram.  
%    % https://q.uiver.app/#q=WzAsNCxbMCwwLCJcXEdhbW1hIl0sWzAsMSwiXFxwaV8xXntcXG9yYn0oXFxoYXRcXGNPKSJdLFsxLDEsIlxccGlfMV57XFxvcmJ9KFxcY08pIl0sWzEsMCwiXFxwaV8xXntcXG9yYn0oXFxjWCkiXSxbMSwyXSxbMywyXSxbMCwzXSxbMCwxXSxbMCwyLCIiLDEseyJzdHlsZSI6eyJuYW1lIjoiY29ybmVyIn19XV0=
%    \[\begin{tikzcd}
%        \Gamma & {\pi_1(\cX)} \\
%        {\pi_1(\hat\cO)} & {\pi_1(\cO)}
%        \arrow[from=1-1, to=1-2]
%        \arrow[from=1-1, to=2-1]
%        \arrow["\lrcorner"{anchor=center, pos=0.125}, draw=none, from=1-1, to=2-2]
%        \arrow[from=1-2, to=2-2]
%        \arrow[from=2-1, to=2-2]
%    \end{tikzcd}\]
Then $\hat\cX = \tilde\cX / \Gamma$ is the desired cover of $\mathcal X$, where $\tilde\cX$ is the universal orbifold cover of $\cX$, and $\Gamma<\pi_1(\mathcal X)$ acts on $\tilde\cX$ by deck transformations.

We apply the preceding observation to obtain finite covers of Seifert orbifolds that have orientable base. 

We say that $\cX$ is \emph{totally orientable} if $\cO$ is orientable (equivalently $\cO$ has only cone point singularities and the underlying surface $|\cO|$ is orientable). If $\cX$ is totally orientable, then the orbifold locus of $\cX$ is a union of fibers lying over a subset of the cone points of $\cO$. Every fiber of $\pi$ is homeomorphic to $S^1$, and all components of $\partial\cX$ are tori. The underlying topological space $|\cX|$ is a 3-manifold and admits a Seifert fibration induced by $\pi$.

Every compact orientable Seifert orbifold $\cX$ has a totally orientable double cover $\cX^+ \rightarrow \cX$, obtained by pullback of the orientation cover $\cO^+\to\cO$ of the base orbifold. 
For this cover, each torus in $\partial\cX$ is covered by two disjoint copies in $\partial \cX^+$, and each pillow in $\partial\cX$ is covered by its torus double cover.

\subsection{Relative Euler number of totally orientable Seifert orbifolds} \label{sec:relative-euler}

In this section, we define the relative Euler number of a totally-orientable Seifert orbifold $\cX$ following Luecke--Wu \cite{luecke-wu}. The definition depends on a choice of framing on $\partial\cX$, which is a union of oriented simple closed curves $\mf a\subset\partial\cX$, one for each boundary component, none of which is isotopic to a fiber. The pair $(\cX,\mf a)$ is called a framed Seifert orbifold. Since each component of $\partial\cX$ is a torus (because $\cX$ is totally orientable), we can encode $\mf a$ by a homology class of $\partial\cX$. 

The relative Euler number of a Seifert fibered \emph{manifold} $(X, \mf a)$ (cf.\ \cite{luecke-wu}) is defined as $e(X,\mf a):=e(X(\mf a))$ where $X(\mf a)$ is the Dehn filling and $e(-)$ is the (usual) Euler number for a closed Seifert fibered 3-manifold; see e.g.\ \cite{scott}. Generalizing this, when $(\cX,\mf a)$ is a totally orientable Seifert fibered 3-orbifold, we define $e(\cX,\mf a):=e(|\cX|,\mf a)$, where $|\cX|$ denotes the underlying Seifert fibered 3-manifold.

\section{Geometric extension of flat boundary-metrics} \label{sec:extend}

In this section $\cX$ is a compact 3-orbifold with boundary that is either atoroidal or Seifert fibered and admits a nonpositively curved metric. In terms of geometrization, $\cX$ is geometric, and admits one of the geometries $\HH^3$, $\EE^3$, or $\HH^2\times\R$. 

We address the following problem: given a flat metric on $\partial\cX$, does this metric extend to a nonpositively curved metric  on $\cX$ for which the boundary is totally geodesic? If so, we say the flat metric on $\partial\cX$ is \emph{geometrically extendable}. This is a 3-dimensional version of the problem for 2-orbifolds discussed in Section \ref{sec:boundary-lengths}. 

Our work here builds on work of Leeb \cite{leeb} in the manifold case. We also refer the reader to \cite{byron} which includes a more detailed account of Leeb's work. 

We discuss the cases when $\cX$ is atoroidal and Seifert fibered separately. 

\subsection{Atoroidal orbifolds} \label{sec:atoroidal-extend}

For atoroidal $\cX$, flat boundary-metrics are always geometrically extendable. 

\begin{prop} \label{prop:atoroidal}
Let $\mathcal X$ be a compact atoroidal $3$-orbifold. Any flat metric $\mu_{\partial}$ on $\partial \mathcal X$ extends geometrically to $\mathcal X$. Furthermore, one can choose the extending metric so that it has the form $\mu_{\partial} + dt^2$ on a collar neighborhood of $\partial \mathcal X$. 
\end{prop}

For atoroidal 3-manifolds, Proposition \ref{prop:atoroidal} is proved by Leeb \cite[Prop.\ 2.3]{leeb}. Our proof of Proposition \ref{prop:atoroidal} is a modified version of Leeb's argument. First we recall that by geometrization  \cite{boileau-leeb-porti}, if $\cX$ is compact and atoroidal, then $\Int(\mathcal{X})$ admits a finite-volume complete hyperbolic metric. Given this, the idea is to interpolate between the complete finite-volume hyperbolic metric on $\Int(\mathcal{X})$ and the flat metric on $\partial X$ so that the resulting metric remains nonpositively curved. In Leeb's case, $\partial X$ is a union of tori and Leeb shows that one can interpolate between any two flat metrics on $T^2$ over a cylindrical region $T^2 \times I$ so that the resulting metric on $T^2 \times I$ is nonpositively curved. We extend this to the remaining closed orientable Euclidean 2-orbifolds.

\begin{proof}[Proof of Proposition \ref{prop:atoroidal}]
    Let $\nu$ be the complete finite-volume hyperbolic metric of constant curvature $-1$ on $\Int(\mathcal{X})$. 
    
For each component $\mathcal{T}$ of $\partial \mathcal{X}$, choose a regular neighborhood $\mathcal N$ of $\mathcal T$ in $\cX$ so that the hyperbolic metric $\nu$ on $\mathcal N\cap\mathrm{int}(\mathcal X)$ has the form $e^{-2t} \nu_{\mathcal{T}}+dt^2$, where $\nu_{\mathcal T}$ is a flat metric on $\mathcal T$, and this expression is with respect to some coordinates $\mathcal{N}\cong\mathcal T\times[r,\infty]$, where $\mathcal T\times\{\infty\}$ is identified with the boundary component $\mathcal T$.

To prove the proposition, it suffices to find a constant $c=c(\cX)$ and to find, for each component $\mathcal T$ of $\partial\cX$, a nonpositively curved metric on $\mathcal T\times[r,\infty]$ that smoothly interpolates between $e^{-2t}\nu_{\mathcal T}+dt^2$ near $\mathcal T\times\{r\}$ and $c^2\mu_{\mathcal T}+dt^2$ near $\mathcal T\times\{\infty\}$, where $\mu_{\cT}$ is the restriction of $\mu_\partial$ to $\cT$. Indeed, these metrics assemble to give a nonpositively curved metric $\mu$ on $\cX$ and $\mu/c^2$ is the desired nonpositively curved metric that agrees with $\mu_\partial$ on $\partial\cX$.

\paragraph{\boldmath Cusp interpolation between flat metrics.}

As above, fix $\mathcal T$ with the hyperbolic metric $\nu=e^{-2t}\nu_{\mathcal T}+dt^2$ on $\mathcal T\times[r,\infty)$. First we modify $\nu$ to equal 
$e^{-2t}\mu_{\mathcal T}+dt^2$ near $t=\infty$ without changing the metric near $t=r$.

By Lemmas \ref{lem:turnover} and \ref{lem:torus-pillow}, there exists $\phi \in \Diff_0(\mathcal{T})$ such that  $\phi^* \nu_{\mathcal{T}}$ and $\mu_{\mathcal{T}}$ are affinely equivalent. In the turnover case this means that $\phi^* \nu_{\mathcal{T}} = \alpha^2\>\mu_{\mathcal{T}}$ for some constant $\alpha>0$. We use a smooth isotopy from $\Id_{\mathcal{T}}$ to $\phi$ to construct a level-preserving diffeomorphism $\Phi$ of $\mathcal T\times[r,\infty)$ that is the identity near $\mathcal T\times\{r\}$ and is $\phi$ near $\mathcal T\times\{\infty\}$. The metric $\Phi^*\nu$ on $\mathcal T\times[r,\infty)$ equals $\nu$ near $\mathcal T\times\{r\}$ and equals $e^{-2t}\phi^*\nu_{\mathcal T}+dt^2$ on $\mathcal T\times[R,\infty)$ for some $R>r$. Pulling back to the universal cover $p:\R^2\times[R,\infty)\to \mathcal T\times[R,\infty)$ and using Remark \ref{rmk:normal-form}, our constructed metric $p^*(\Phi^*\nu)$ is equal to $e^{-2t}(dx^2+dy^2)+dt^2$, while our target metric $p^*(e^{-2t}\mu_{\mathcal T}+dt^2)$ has the form $e^{-2t}(a^2\>dx^2+b^2\>dy^2)+dt^2$ for some constants $a,b>0$. In the turnover case $a=b$. To interpolate between these,  let $\psi:[R,\infty)\to[0,1]$ be a smooth function that equals 0 near $R$ and $1$ near $\infty$ and consider the modified metric on $\R^2\times[R,\infty)$ defined by 
\begin{equation}\label{eqn:warp1}e^{-2t}((1 - \psi(t))+\psi(t)a)^2 \cdot dx^2 + e^{-2t} ( (1 - \psi(t))+\psi(t)b)^2 \cdot dy^2+dt^2. \end{equation}
This metric is equal to our constructed metric near $\R^2\times\{R\}$ and is equal to (the pullback of) our desired metric near $\R^2\times\{\infty\}$. If we arrange (as we can) for the functions $e^{-t}((1 - \psi(t))+\psi(t)a)$ and $e^{-t} ( (1 - \psi(t))+\psi(t)b)$ to be strictly monotonically decreasing and convex, then this metric is nonpositively curved by a straightforward computation (cf.\ e.g.\ \cite[pg.\ 26--27]{bishop-oneill}).

To descend to $\mathcal T\times[R,\infty)$, note that the metrics in (\ref{eqn:warp1}) are invariant under the deck group action $\pi_1(\mathcal T)\curvearrowright\R^2$. Indeed, in the torus and pillow cases, $dx^2$ and $dy^2$ are invariant; in the turnover case $dx^2+dy^2$ is invariant and we use that $a=b$. Thus the metric (\ref{eqn:warp1}) descends to a metric on $\mathcal T\times[R,\infty)$.  The concatenation of this metric with $\Phi^*\nu$ on $\mathcal T\times[r,R]$ is a nonpositively curved metric that equals $\nu=e^{-2t}\nu_{\mathcal T}+dt^2$ near $\mathcal T\times\{r\}$ and equals $e^{-2t}\mu_{\mathcal T}+dt^2$ on $\mathcal T\times[q,\infty)$ for some $q>r$. 

\paragraph{Cusp flattening.} 
Here we observe that given the nonpositively curved metric $e^{-2t}\mu_{\mathcal T}+dt^2$ on $\mathcal T\times[q,\infty)$, for $c>0$ sufficiently small, we can find a nonpositively curved metric on $\mathcal T\times[q,\infty)$ that equals the given metric near $t=q$ and is equal to $c^2\mu_{\mathcal T}+dt^2$ near $t=\infty$. For this, fix $c<e^{-q}$ sufficiently small so that there is a non-increasing convex function $\rho: [q, \infty) \rightarrow \R^+$ such that $\rho(t) = e^{-t}$ near $q$ and $\rho(t) \equiv c$ near $\infty$, and consider the metric $\rho^2(t) \mu_{\mathcal{T}}+dt^2$. The assumption that $\rho$ is non-increasing and convex guarantees this metric is nonpositively curved \cite[Rmk.\ 7.9]{bishop-oneill}. 

Concatenating the metric $\rho^2(t) \mu_{\mathcal{T}}+dt^2$ on $\mathcal T\times[q,\infty)$ with the metric constructed on $\mathcal T\times[r,q]$ in the previous step, taking $c$ small enough to work for each component $\mathcal T$ of $\partial\cX$, and truncating the ends (so as to get a metric on a compact orbifold diffeomorphic to $\cX$) finishes the proof.
\end{proof}

\subsection{Seifert fibered orbifolds} \label{sec:seifert-extend}

We continue to study the problem of extending  flat boundary-metrics geometrically. For Seifert fibered orbifolds, unlike atoroidal orbifolds, geometric extensions do not always exist. Our analysis in the Seifert fibered case hinges on totally-orientable Seifert orbifolds. For these we have a homological criterion (Proposition \ref{prop:leeb}) generalizing Leeb \cite{leeb} in the manifold case. Every Seifert fibered orbifold has a totally-orientable double cover (\S\ref{sec:TO-cover}) and we show that the metric extension problem can be reduced to this cover (Proposition \ref{prop:totally-orientable-cover}) or to any further cover (Proposition \ref{prop:cover-totally-orientable}).  

\subsubsection{Homological criterion for extending geometrically} \label{sec:leeb-condition}

In this section $\cX$ is a totally-orientable Seifert fibered orbifold with nonempty boundary. We also assume that $\cX$ admits a nonpositively curved metric; equivalently $\cX$ is modeled by either $\EE^3$ or $\HH^2 \times \RR$. 

Proposition \ref{prop:leeb} gives a (co)homological criterion for when a flat boundary-metric on $\cX$ is geometrically extendable when $\cX$ is totally orientable. It reformulates and extends a result of Leeb, as explained in Remark \ref{rmk:leeb}. 

To state Proposition \ref{prop:leeb}, recall from Section \ref{sec:torus-pillow} (cf.\ (\ref{eqn:inner-product})) that a flat metric $\nu$ on $T^2$ determines an inner product $B_\nu$ on $H_1(T^2)\otimes\R$. If $T^2$ is a boundary component of a Seifert orbifold, then the fiber class $f$ (well-defined up to $\pm1$) gives a function 
\[\begin{array}{rcl}
\beta_\nu:H_1(T^2;\R)&\to&\R\\
x&\mapsto&B_\nu(x,f),
\end{array}\]
which we view as an element of $H^1(T^2;\R)$. Accordingly, if $\cX$ is a totally-orientable Seifert fibered orbifold, then a flat boundary metric $\mu_\partial$ on $\cX$, together with a choice of regular fiber of $\cX$, determines a class $\beta_{\mu_\partial}\in H^1(\partial\cX;\R)$. 

In the following statement the cohomology of an orbifold $\cX$ is the (singular) (co)homology of its underlying topological space $X=|\cX|$. 

\begin{prop}[Homological criterion] \label{prop:leeb}
Let $\cX$ be a totally-orientable Seifert fibered $3$-orbifold, and assume $\cX$ is nonpositively curved. Fix a flat metric $\mu_{\partial}$ on $\partial\cX$, and write $\beta_{\mu_\partial}\in H^1(\partial\cX;\R)$ for the associated class defined above. If $\cX\neq T^2\times[0,1]$, then $\mu_{\partial}$ is geometrically extendable if and only if $\beta_{\mu_\partial}$ is in the image of the restriction map 
\[H^1(\cX;\R)\to H^1(\partial\cX;\R).\] 
If $\cX=T^2\times[0,1]$ and the boundary inner products are $B_0,B_1$, then $\mu_\partial$ is geometrically extendable if and only if $B_0=B_1$. Furthermore, in every case, the extending metric may be chosen to be flat near $\partial\cX$.
\end{prop}

\begin{proof}[Proof of Proposition \ref{prop:leeb}]
    The proof is ultimately similar to the argument of \cite[Lem.\ 2.5]{leeb} (see also \cite[Prop.\ 2.8]{byron}), even though our formulation in terms of cohomology is different. 

\paragraph{Forward direction.} Assume $\mu_\partial$ extends to $\mu$ as in the statement. As explained by Leeb \cite[\S2.3]{leeb}, the universal cover of $(\cX,\mu)$ splits as a Riemannian product $Y\times\R$ in such a way that the vertical copies of $\R$ project to fibers of the given Seifert fibration. 
(The $\R$-factor corresponds to the central subgroup generated by a regular fiber.) 
The action of the deck group $\pi_1(\cX)$ on $Y\times\R$ induces a homomorphism $\rho_\R:\pi_1(\cX)\to\Isom(\R)$ whose image is contained in $\Isom^+(\R)$ because $\cX$ is totally orientable. Since $\Isom^+(\R)\cong\R$ is abelian, $\rho_\R$ factors through a homomorphism 
\[\widehat\beta:H_1(\cX;\R)\to\R.\]
To obtain $\widehat\beta$, we use the isomorphisms
\[\pi_1(\cX)^{\mathrm{ab}}\otimes\R\cong\pi_1(X)^{\mathrm{ab}}\otimes\R\cong H_1(X;\R)= H_1(\cX;\R),\] where $X=|\cX|$ is the underlying topological space (manifold) of $\cX$. The first isomorphism can be seen from the standard presentations of $\pi_1(\cX)$ and $\pi_1(X)$. Specifically, $\pi_1(\cX)$ has a presentation $\pi_1(\cX) = \langle S \mid R \rangle$, where  
    \begin{equation*}
        S=
        \left\{
            b_1,\ldots,b_n,
            c_1,\ldots,c_m,
            x_1,y_1,\ldots,x_g,y_g,
            f
        \right\}
    \end{equation*}
    and
    \begin{equation*}
        R=
        \left\{
            \text{$f$ central}, \, 
            c_j^{p_j} = f^{q_j}, \, 
            \prod_{i=1}^n b_i
            \prod_{j=1}^m c_j
            \prod_{k=1}^g [x_k, y_k]
            = f^{e_0}
        \right\}.
    \end{equation*}
Here $g$ is the genus of the base orbifold, $f$ corresponds to a regular fiber, and the $b_i$ and $c_j$ correspond to boundary components and cone points in the base orbifold. Unlike the manifold case, $p_j,q_j$ need not be relatively prime because $\cX$ may have nontrivial isotropy along its singular fibers. If we set $d_j = \gcd(p_j,q_j)$, $p_j' = p_j / d_j$, and $q_j' = q_j / d_j$, then $c_j^{p_j'} f^{-q_j'}$ is the meridian of the underlying solid-torus neighborhood of the $j$-th singular fiber. In the orbifold fundamental group, this element has order $d_j$, corresponding to the local isotropy group $\ZZ/d_j\ZZ$. 
For the underlying topological space $X = |\cX|$, the loops $c_j^{p_j'} f^{-q_j'}$ are trivial, and $\pi_1(X)$ has a similar presentation with relations $c_j^{p_j} = f^{q_j}$ replaced by $c_j^{p_j'} = f^{q_j'}$. For more information, see e.g.\ \cite{scott}. 

The restriction of $\rho_\R:H_1(X)\to\R$ to $H_1(\partial X)$ is given by $x\mapsto B(x,f)/|f|$, where $|f|=B(f,f)^{1/2}$ and $B$ is the inner product associated to $\mu_\partial$. This proves that $\beta_{\mu_\partial}\in H^1(\partial X)$ is the image of $|f|\cdot \rho_\R$ under $H^1(\cX;\R)\cong H^1(X;\R)\to H^1(\partial X;\R)\cong H^1(\partial\cX;\R)$.

\paragraph{Reverse direction.} 
We construct a discrete faithful representation 
\[\rho=\rho_{Y}\times\rho_{\R}: \pi_1(\cX) \to \Isom^+(Y)\times\Isom^+(\RR)\] that will give $\cX$ a nonpositively curved metric that geometrically extends $\mu_\partial$. Here we take $Y = \EE^2$ or $\mathbb H^2$ depending on whether $\chi(\cO) $ is zero or negative. 

We define $\rho_{Y}$ by factoring through a geometric representation $\pi_1(\cO)\to\Isom(Y)$ of the base orbifold $\cO$. The values of the translation lengths of $\rho_{Y}$ on the boundary curves $b_i$ are constrained by the given flat metric on $T_i$. Specifically, the translation length of $\rho_{Y}(b_i)$ equals $\big[ B_i(b_i,b_i) - B_i(b_i,f)^2 / |f|^2 \big]^{1/2}$ by the Pythagorean theorem. The existence of $\rho_{Y}$ follows from Lemma \ref{lem:boundary-lengths} and the fact that $\cO$ is orientable and not orbifold-diffeomorphic to $S^1 \times [0,1]$ (because $\cX\neq T^2\times[0,1]$). 

To define $\rho_{\RR}$, we use that $\beta_{\mu_\partial}:H_1(\partial\cX;\R)\to\R$ extends to a homomorphism $\beta:H_1(\cX;\R)\to\R$, and we define $\rho_\R$ as the composition 
\[
\pi_1(\cX)\to\pi_1(\cX)^{\mathrm{ab}}\otimes\R\cong\pi_1(X)^{\mathrm{ab}}\otimes\R\cong H_1(X;\R)=H_1(\cX;\R)\xrightarrow{\beta/|f|}\R.\] By construction $\rho_{\RR}(f) = B(f,f)^{1/2}$ and $\rho_{\RR}(b_i) = B(b_i, f)/|f|$, so the metric defined by $\rho=\rho_Y\times\rho_\R$ defines a geometric extension of $\mu_\partial$.

\paragraph{\boldmath $T^2\times[0,1]$.} 
If $\mu_\partial$ geometrically extends to $\mu$, then $\mu$ is metrically a product. Indeed, in the universal cover the boundary components are convex and the distance function $d_{\tilde\mu}(x,\tilde T_1)$ on $\tilde T_0$ is uniformly bounded (because $T^2$ is compact) and convex (true generally for CAT(0) geometry \cite[Ch.\ II.2, Cor.\ 2.5]{bridson-haefliger}), and this implies the universal cover is metrically a product \cite[Ch.\ II.2, Exercise 2.12(2)]{bridson-haefliger}. Since the deck group acts isometrically, this implies $\mu$ is a product metric. Thus $B_0=B_1$. 

Conversely, if $B_0=B_1$, then the boundary metrics are isotopic, i.e.\ there exists $\phi\in\Diff_0(T^2)$ such that $\phi^*(\mu_\partial\mid_{T^2\times\{0\}})=\mu_\partial\mid_{T^2\times\{1\}}$. Choose a smooth isotopy $\phi_t$ such that $\phi_t=\id$ near $0$ and $\phi_t=\phi$ near $1$. Define a diffeomorphism $\Phi:T^2\times[0,1]\to T^2\times[0,1]$ by $(x,t)\mapsto (\phi_t(x),t)$. Then the metric $\Phi^*(\mu_\partial\mid_{T^2\times\{0\}}\times dt^2)$ has the desired form.

Finally, the ``furthermore" statement of the proposition is immediate if $\cX$ has $\mathbb E^3$ geometry. If $\cX$ has $\HH^2\times\R$ geometry, then the statement can be proved in a similar way to the cusp-flattening step in the proof of Proposition \ref{prop:atoroidal} (flattening the cusp in the $\HH^2$ direction). 
\end{proof}

\begin{rmk}\label{rmk:leeb}
The condition that $\beta_{\mu_\partial}:H_1(\partial\cX;\R)\to\R$ extends to $H_1(\cX;\R)\to\R$ can be expressed in coordinates as follows. From the presentation of $\pi_1(\cX)$ in the proof of Proposition \ref{prop:leeb}, we see that $H_1(\cX;\R)$ has a presentation (as a vector space) with generating set $b_1,\ldots,b_n,f_1,\ldots,f_n, x_1,y_1,\ldots,x_g,y_g$ and relations $\sum b_i= e(\cX,\kb) f_1$ and $f_i=f_1$ for each $1\le i\le n$. Here $\kb$ is the framing of $\partial\cX$ determined by the $b_i$, and $f_i$ is a regular fiber on the $i$-th component of $\partial\cX$. By the definition of the relative Euler number (see Section \ref{sec:relative-euler}), we have 
\[e(\cX, \kb)= e(X, \kb)= e_0 - \sum_{j=1}^m \frac{q_j'}{p_j'}= e_0 - \sum_{j=1}^m \frac{q_j}{p_j}.\]
Given this presentation, we see that $\beta_{\mu_\partial}:H_1(\partial\cX;\R)\to\R$ extends to $H_1(\cX;\R)\to\R$ if and only if 
\begin{enumerate}[(1)]
\item (fiber length) $B(f_i,f_i)=B(f_1,f_1)$ for each $i$, and 
\item (relative Euler number) $\sum_i B(b_i,f_i)=e(\cX,\kb)B(f_1,f_1)$. 
\end{enumerate} 
This aligns with the form of the homological criterion that appears in \cite{leeb, buyalo-svetlov, byron}. We refer to conditions (1) and (2) as \emph{Leeb's conditions}. 
\end{rmk}

\begin{rmk}\label{rmk:convex-combination}
The preceding equations (1) and (2) are a system of linear equations on the inner product $B$. Since the set of inner products on a real vector space forms a convex cone, so too does the set of inner products satisfying Leeb's conditions. 
\end{rmk}

\subsubsection{Reduction to totally-orientable double covers} 

The following proposition reduces the problem of determining whether a flat metric on $\partial\cX$ extends geometrically to the same problem on the totally-orientable double cover. 

\begin{prop}\label{prop:totally-orientable-cover}
Let $\mathcal X$ be a nonpositively curved Seifert fibered $3$-orbifold not diffeomorphic to $\cT \times [0,1]$ where $\cT$ is a torus or a pillow. Let  $p:\mathcal X^+\to\mathcal X$ be the totally orientable double cover. Then a flat metric $\mu_\partial$ on $\partial\mathcal X$ extends geometrically to $\mathcal X$ if and only if $p^*(\mu_\partial)$ extends geometrically  to $\mathcal X^+$. 
\end{prop}

\begin{proof}[Proof of Proposition \ref{prop:totally-orientable-cover}]

First we remark that the case $\cX$ is already totally orientable is trivial, so we assume $\cX$ is not totally orientable. Also, the forward direction is obvious: if $\mu_\partial$ extends to $\mu$, then $p^*(\mu)$ extends $p^*(\mu_\partial)$.

 For the reverse direction, assume $p^*(\mu_\partial)$ extends geometrically to $\cX^+$. To prove the proposition, we construct a representation $\rho:\pi_1(\cX)\to\Isom^+(Y\times\R)$, where $Y$ is either $\mathbb H^2$ or $\mathbb E^2$, that defines a nonpositively curved metric on $\cX$ that extends $\mu_\partial$. Such a representation $\rho$ is determined by representations $\rho_Y:\pi_1(\cX) \to \Isom(Y)$ and $\rho_\R:\pi_1(\cX)\to\Isom(\R)$, which we construct separately. As in the proof of the homological criterion (Proposition \ref{prop:leeb}) the main subtlety is in defining $\rho_\R$. 

First we explain how to construct $\rho_Y$. The boundary metric $\mu_\partial$ imposes lengths on the boundary components of the base orbifold $\cO$ of $\cX$. We choose $\rho_Y$ to factor as $\pi_1(\cX)\to\pi_1(\cO)\to\Isom(Y)$, where the latter map defines a geometric structure on $\cO$ with the prescribed  boundary lengths. This is possible by Lemma \ref{lem:boundary-lengths} and that fact that $O \not\cong S^2 \times [0,1]$ or $\bar I \times [0,1]$ (since $\cX \not\cong \cT \times [0,1]$).

The restriction of $\rho_Y$ to $\pi_1(\cX^+)$ is a representation $\rho_Y^+:\pi_1(\cX^+)\to\Isom^+(Y)$ 
that realizes the boundary lengths of $\cO^+$ determined by $p^*\mu_\partial$. By the assumption that $p^*\mu_\partial$ extends geometrically and the homological criterion (Proposition \ref{prop:leeb}), which provides an extension by a geometric metric, there is a representation $\rho_\R^+:\pi_1(\cX^+)\to\Isom^+(\R)$ so that 
\[\rho_Y^+\times \rho_\R^+:\pi_1(\cX^+)\to\Isom^+(Y)\times \Isom^+(\R)\]
defines a nonpositively curved metric on $\cX^+$ that extends $p^*\mu_\partial$. 

Now we define $\rho_\R:\pi_1(\cX)\to\Isom(\R)$ so that 
\[\rho_Y\times\rho_\R:\pi_1(\cX)\to\Isom(Y)\times\Isom(\R)\]
defines a nonpositively curved metric on $\cX$ that extends $\mu_\partial$. For this we require: the restriction of $\rho_\R$ to $\pi_1(\cX^+)$ is contained in $\Isom^+(\R)$; $\rho_\R(f)\neq0$, where $f$ is the fiber class; $\rho_\R(y)$ is a reflection for every (equivalently some) $y$ in $\pi_1(\cX)\setminus\pi_1(\cX^+)$. 

In order for the metric defined by $\rho_Y\times\rho_\R$ to extend $\mu_\partial$, we require additionally that for $x\in\pi_1(\cX^+)$ freely homotopic into $\partial\cX^+$, 
\[\rho_\R(x)=\frac{B_{p^*\mu_\partial}(x,f)}{B_{p^*\mu_\partial}(f,f)^{1/2}}=\rho_\R^+(x).\] 
The second equality holds by definition of $\rho_\R^+$. 

Define $\rho_\R$ as follows. Fix an element $y$ in $\pi_1(\cX)\setminus\pi_1(\cX^+)$ and let $\tau$ be the automorphism of $\pi_1(\cX^+)$ induced by conjugation by $y$. 
\begin{itemize}
\item $\rho_\R(y)$ is the reflection $r\in\Isom(\R)$ about the origin $[t\mapsto -t]$
\item for $x\in\pi_1(\cX^+)$, \[
\rho_\R(x)=\frac{1}{2}\big[\rho_\R^+(x)-\rho_\R^+(\tau x)\big].
\]
\end{itemize}

We check that $\rho_\R$ defines a homomorphism. This is clear for the restriction of $\rho_\R$ to $\pi_1(\cX^+)$. Since $\pi_1(\cX)$ contains $\pi_1(\cX^+)$ as an index-2 normal subgroup, it has a presentation with generating set $\pi_1(\cX^+)$ and $y$ and whose relations are the relations in $\pi_1(\cX^+)$ and the relations $yxy^{-1}=\tau(x)$ for $x\in\pi_1(\cX^+)$ and the relation $y^2=w$ for some $w\in\pi_1(\cX^+)$. We need to check that these last two relations are preserved by $\rho_\R$. 
For $yxy^{-1}=\tau(x)$, we compute
\[
\rho_\R(y)\rho_\R(x)\rho_\R(y)^{-1}=r\rho_\R(x) r^{-1}=-\rho_\R(x)
\]
and
\[
\rho_\R(\tau x)=\frac{1}{2}\big[\rho_\R^+(\tau x)-\rho_\R^+(\tau^2x)\big]=-\rho_\R(x),
\]
observing that $\tau^2x=y^2xy^{-2}=wxw^{-1}$ and that $\rho_\R^+$ factors through the abelianization. For the relation $y^2=w$, we compute
\[\rho_\R(y)^2=r^2=0\>\>\>\text{ in }\R=\Isom^+(\R)
\]
and
\[\rho_\R(w)=\frac{1}{2}\big[\rho_\R^+(w)-\rho_\R^+(\tau w)\big]=0
\]
with the last equality holding because $\tau(w)=ywy^{-1}$ equals $w$ since $y$ commutes with $w=y^2$. This shows $\rho_\R$ is a homomorphism. 

Finally, we check that $\rho_\R=\rho_\R^+$ on $\partial\cX^+$. By definition of $\rho_\R$, it is equivalent to show that $\rho_\R^+(\tau x)=-\rho_\R^+(x)$. Here we use that, by construction, $p^*\mu_\partial$ is invariant under the deck transformation of the cover $\cX^+\to\cX$ (which acts on $\pi_1(\cX^+)$ by $\tau$). Given this, we compute for $x\in\pi_1(\cX^+)$ freely homotopic to the boundary
\[\rho_\R^+(\tau x)=\frac{B_{p^*\mu_\partial}(\tau x,f)}{|f|}=\frac{B_{p^*\mu_\partial}(x,\tau f)}{|f|}=\frac{B_{p^*\mu_\partial}(x,-f)}{|f|}=-\rho_\R^+(x).\]
Here we also use that $\tau(f)=f^{-1}$ since $\cX$ is not totally orientable. 
\end{proof}

\subsubsection{Geometric extendability with a free boundary component}

For later use, we record the following version of Proposition \ref{prop:leeb} in the case when the boundary metric is not specified on the entire boundary. 

\begin{lem} \label{lem:free-boundary}
    Let $\pi: \cX \rightarrow \cO$ be a compact orientable Seifert fibered $3$-orbifold with nonempty boundary, and fix a boundary component $\mathcal S$. A flat metric $\mu_{\partial}$ on $\partial\cX\setminus\mathcal S$ is geometrically extendable on $\cX$ if and only if all regular fibers of $\partial\cX\setminus\mathcal S$ have the same $\mu_{\partial}$-length. 
\end{lem}

When $\cX$ is a manifold and $\chi(\cO) < 0$, Lemma \ref{lem:free-boundary} follows from \cite[Prop.\ 2.6]{leeb}.

\begin{proof}[Proof of Lemma \ref{lem:free-boundary}]
The forward direction follows from the fact that all regular fibers of a Seifert fibered orbifold equipped with a compatible nonpositively curved metric have the same length. 

For the reverse direction, we begin with the case $\cX$ is totally orientable. We want to choose a flat metric on $\mathcal S$ whose inner product $B_{\mathcal S}$ (together with the inner products on the other boundary components coming from $\mu_\partial$) satisfies the linear equations in Leeb's condition (Remark \ref{rmk:leeb}). Fixing a basis $b_{\mathcal S},f_{\mathcal S}$ for $H_1(\mathcal S)$ with $f_{\mathcal S}$ the fiber direction, the values $B_{\mathcal S}(f_{\mathcal S},f_{\mathcal S})$ and $B_{\mathcal S}(b_{\mathcal S},f_{\mathcal S})$ are determined by the fiber length and relative Euler number constraints in Remark \ref{rmk:leeb}. If $\cX\neq T^2\times[0,1]$, we only need to choose the value $B_{\mathcal S}(b_{\mathcal S},b_{\mathcal S})>0$ sufficiently large so that $B_{\mathcal S}$ is positive definite. If $\cX=T^2\times[0,1]$, we choose the two boundary inner products to agree.

Assume now that $\cX$ is not totally orientable, and let $p: \cX^+ \rightarrow \cX$ be the totally-orientable double cover. It suffices to extend $p^*\mu_\partial$ to a $\tau$-invariant flat metric $\nu$ on all of $\partial\cX^+$ that extends geometrically to $\cX^+$. Then by $\tau$-invariance, $\nu$ has the form $p^*(\widehat{\mu_\partial})$ for a flat metric $\widehat{\mu_\partial}$ that extends the metric $\mu_\partial$ on $\partial\cX\setminus\mathcal S$ to $\partial\cX$. By construction 
$p^*(\widehat{\mu_\partial})=\nu$ extends geometrically and this implies that $\widehat{\mu_\partial}$ extends geometrically by Proposition \ref{prop:totally-orientable-cover}.

We proceed to construct the flat metric $\nu$ on $\partial\cX^+$. By the preceding step, we can find a flat metric $\nu^{npc}$ on $\partial\cX^+$ that agrees with $p^*\mu_\partial$ on $\partial\cX^+\setminus p^{-1}(\mathcal S)$ and extends geometrically to $\cX^+$ (there is a slight difference in that $p^{-1}(\mathcal S)$ can have two components, but this does not present any additional difficulty). Choose also a flat metric $\nu^\tau$ that extends $p^*\mu_\partial$ and additionally is invariant under the nontrivial deck transformation $\tau$ of the cover $p$. By Lemma \ref{lem:torus-pillow}, there is a (unique) flat metric $\nu_0$ on $\partial\cX^+$ that is isotopic to $\nu^{npc}$ and affinely equivalent to $\nu^\tau$. Define $\nu=\frac{1}{2}(\nu_0+\tau^*\nu_0)$, 
which is $\tau$ invariant by construction. The metric $\nu$ is flat: $\nu_0$ and $\tau^*(\nu_0)$ are affinely equivalent, each being affinely equivalent to $\nu^\tau$, so the average is also flat (cf.\ Remark \ref{rmk:average}). Finally, $\nu$ extends geometrically to $\cX^+$ because $\nu_0$ extends geometrically (being isotopic to $\nu^{npc}$), as does $\tau^*\nu_0$, and this implies their average extends geometrically by Remark \ref{rmk:convex-combination}. 
\end{proof}

\begin{rmk}
An argument similar to the averaging argument in the last paragraph of the proof of Lemma \ref{lem:free-boundary} will appear again in the proof of Proposition \ref{prop:average-JSJ}.
\end{rmk}

Note that Leeb's condition (Remark \ref{rmk:leeb}) implies that, when $\cX \not\cong \cT \times [0,1]$, the geometric extendability of a flat boundary metric $\mu_\partial$ on $\partial\cX$ is independent of the metric in the direction orthogonal to a regular fiber. Thus, we may freely rescale $\mu_\partial$ in the direction orthogonal to the fibers while preserving geometric extendability. This observation yields Corollary \ref{cor:orthogonal-scaling}, which we will use in the proof of Theorem \ref{thm:main} when the JSJ decomposition of $M$ has an atoroidal component. 

\begin{cor} \label{cor:orthogonal-scaling} 
Let $\cX$ be a compact orientable Seifert fibered $3$-orbifold with nonempty boundary.  Assume $\cX$ is not $\cT \times [0,1]$ for $\cT$ a torus or pillow. Fix a nonempty union $\mathcal Z$ of boundary components  of $\cX$, fix a nonpositively curved metric $\mu_\partial$ on $\partial\cX\setminus\mathcal Z$ with constant fiber length, and fix a nontrivial closed curve $a_{\cT}$ on each component $\mathcal T$ in $\mathcal Z$ such that $a_{\cT}$ is not isotopic to a regular fiber. If $C>0$ is sufficiently large, then $\mu_\partial$ extends geometrically to $\cX$ in such a way that the length of $a_{\cT}$ is $C$ for each component $\cT$ of $\mathcal Z$. 
\end{cor}

\subsubsection{\boldmath Covers and boundary extension}\label{sec:boundary-characteristic}

The following Proposition \ref{prop:cover-totally-orientable} shows that for totally-orientable Seifert orbifolds, to show a flat boundary metric extends geometrically, it suffices to prove this on any cover.

\begin{prop}\label{prop:cover-totally-orientable}
Let $p:\cY\to\cX$ be a covering map between totally orientable Seifert orbifolds. Then a flat boundary-metric $\mu_\partial$ on $\mathcal X$ is geometrically extendable if and only if  $p^*(\mu_\partial)$ is geometrically extendable on $\cY$. 
\end{prop}

\begin{proof}
The forward direction is obvious: if $\mu_\partial$ extends, then $p^*\mu_\partial$ extends via the pullback metric. For the reverse direction, we apply the homological criterion (Proposition \ref{prop:leeb}) and use the transfer in cohomology. Since the statement of Proposition \ref{prop:leeb} depends on whether the orbifold is $T^2\times[0,1]$ or not, we separate these cases. 

If $\cX\neq T^2\times[0,1]$, then the proof follows by considering the following diagram. 
\[\begin{xy}
(-25,0)*+{H^1(\cY;\R)}="A";
(0,0)*+{H^1(\partial\cY;\R)}="B";
(-25,-15)*+{H^1(\cX;\R)}="C";
(0,-15)*+{H^1(\partial\cX;\R)}="D";
{\ar"C";"A"}?*!/^3mm/{p^*};
{\ar "D";"B"}?*!/_3mm/{p^*};
{\ar "C";"D"}?*!/^3mm/{i^*_{\cX}};
{\ar "A";"B"}?*!/_3mm/{i^*_{\cY}};
{\ar@/_5mm/ "A";"C"}?*!/^3mm/{p_!};
{\ar@/^5mm/ "B";"D"}?*!/_3mm/{p_!};
\end{xy}\]

The inner square commutes by functoriality, and the outer square commutes by naturality of the pushforward map; furthermore, $p_!p^*$ is multiplication by $\deg(p)$ (see e.g.\  \cite[\S3.G]{hatcher}). 

The key relationship we need is $\beta_{p^*\mu_\partial}=d_1\>p^*(\beta_{\mu_\partial})$. Indeed, writing $\hat f$ and $f$ for the fiber classes with $p_*(\hat f)=d_1f$, we have 
\[\beta_{p^*\mu_\partial}(x)=B_{p^*\mu_\partial}(x,\hat f)=B_{\mu_\partial}(p_*x,p_*\hat f)=d_1\>B_\mu(p_*x,f)=d_1\>p^*\beta_{\mu_\partial}(x).\]
Consequently, if $p^*\mu_\partial$ extends geometrically, then $p^*(\beta_{\mu_\partial})$ is in the image of $i^*_{\cY}$. Applying $p_!$ then implies $\beta_{\mu_\partial}$ is in the image of $i^*_{\cX}$, so $\mu_\partial$ extends geometrically, as desired.

If $\cX=T^2\times[0,1]$, then without loss of generality we can assume $\cY$ is connected, so $\cY=T^2\times[0,1]$. If $p^*\mu_\partial$ extends geometrically, then the two boundary inner products $B_{p^*\mu_\partial}^0$, $B_{p^*\mu_\partial}^1$ agree, which implies that $B_{\mu_\partial}^0=B_{\mu_\partial}^1$ since $p$ restricts to the same cover on the two boundary components. Thus $\mu_\partial$ extends geometrically. 
\end{proof}

\section{Proof of Main Theorem} 

Recall from the statement of Theorem \ref{thm:main} that $M$ is a nonpositively curved closed, and $G$ is a finite group acting by orientation-preserving diffeomorphisms of $M$. We want to show that $G$ preserves some nonpositively curved metric, or equivalently, that the quotient orbifold $\cM:=M/G$ has a nonpositively curved metric. By the discussion in Section \ref{sec:previous-work}, it suffices to consider the case when the JSJ decomposition of $M$ is nontrivial. 

\subsection{Case: atoroidal components in the JSJ decomposition}\label{sec:atoroidal-proof}

In this section we prove Theorem \ref{thm:main} in the case when the JSJ decomposition of $M$ has at least one atoroidal component. This assumption implies that the JSJ decomposition of $\cM$ also has an atoroidal component, and the following Proposition \ref{prop:nonpositively curved-atoroidal-component}, implies that $\cM$ is nonpositively curved in this case. The proof relies on the results about geometric extensions of flat boundary-metrics proved in Section \ref{sec:extend}. 

\begin{prop}\label{prop:nonpositively curved-atoroidal-component}
Let $\mathcal M$ be an irreducible $3$-orbifold. Assume that the JSJ decomposition of $\mathcal M$ has an atoroidal component. Then $\mathcal M$ has a nonpositively curved metric.
\end{prop}

The basic strategy is to define nonpositively curved metrics on the components of the JSJ decomposition that glue compatibly. Since flat boundary-metrics are always geometrically extendable on atoroidal components (Proposition \ref{prop:atoroidal}), the main concern is Seifert-fibered components. Here we use the flexibility provided by an unprescribed boundary component to inductively construct compatible nonpositively curved metrics on the Seifert fibered components. 

\begin{proof}[Proof of Proposition \ref{prop:nonpositively curved-atoroidal-component}]

    Let $\cZ$ be a JSJ orbifold of $\cM$. By minimality, no component of $\cM \mid \cZ$ is diffeomorphic to $\cT\times[0,1]$, where $\cT$ is either the torus or the pillow. Since $\cM \mid \cZ$ contains an atoroidal component, we may enumerate its Seifert components $\cX_1,\cX_2, \ldots, $ so that each $\cX_i$ has a boundary component adjacent either to an atoroidal component or to some $\cX_k$ with $k > i$ (consider the dual graph to $\cZ$ and order Seifert vertices by decreasing distance to the set of atoroidal vertices, breaking ties arbitrarily). 

We construct compatible nonpositively curved metrics on the Seifert components inductively. Suppose that compatible metrics have been constructed on $\cX_1, \ldots, \cX_{i-1}$. First we modify these metrics so that every boundary metric these prescribe on $\cX_i$ gives its fibers the same length (while maintaining compatibility). For each earlier neighbor $\cX_j$ of $\cX_i$, consider all boundary components of $\cX_j$ glued to $\cX_i$. On each of these components, take the curve corresponding to the fiber of $\cX_i$. By minimality of $\cZ$, these curves are not isotopic to the fiber of $\cX_j$.  By Corollary \ref{cor:orthogonal-scaling} we can adjust the nonpositively curved metric on $\cX_j$ on the boundary components adjacent to $\cX_i$ (leaving the metric on the other boundary components unchanged) so that all the fibers have the same length $C$ for a sufficiently large constant $C$. Specifically, we choose $C$ large enough so that we may apply Corollary \ref{cor:orthogonal-scaling} to each $\cX_j$ incident to $\cX_i$ with $j\le i-1$, and we perform this procedure for each $j$ such that $j\le i-1$ and $\cX_j$ is adjacent to $\cX_i$. Now $\cX_i$ has a partial flat boundary-metric for which the regular fibers of $\cX_i$ all have the same length $C$. In addition, if $\cX_i$ has a pair of boundary components that are adjacent to a single component $\cT$ of $\cZ$, then the fiber directions coming from the two boundary components are linearly independent since $\cZ$ is minimal, and we choose any flat metric on $\cT$ for which the two fiber directions have length $C$ (it is clear that this is possible at the level of inner products). By construction, $\cX_i$ has a boundary component on which no flat metric has yet been prescribed, so Lemma \ref{lem:free-boundary} gives a nonpositively curved metric on $\cX_i$ extending all the prescribed boundary metrics. Repeating this process produces compatible nonpositively curved metrics on all the Seifert fibered components of $\cM \mid \cZ$.

    Next, we construct metrics on the atoroidal components. Choose an arbitrary flat metric on each JSJ component separating two atoroidal components, and use the previously constructed flat metric on each JSJ component separating an atoroidal component from a Seifert component. This prescribes a flat metric on each boundary component of each atoroidal component, and by Proposition \ref{prop:atoroidal} these flat boundary-metrics extend to nonpositively curved metrics on the atoroidal components.

    Finally, by Propositions \ref{prop:atoroidal} and \ref{prop:leeb} the nonpositively curved metrics on the components of $\cM \mid \cZ$ may be chosen to be flat product metrics on collar neighborhoods of their boundaries, so they glue  smoothly to a nonpositively curved metric on $\cM$.
\end{proof}

\subsection{Case: graph manifolds} \label{sec:main-thm-graph}

In this section, we prove Theorem \ref{thm:main} when $M$ is a nonpositively curved graph manifold (with nontrivial JSJ decomposition). To show that $\cM=M/G$ has a nonpositively curved metric, we fix a JSJ orbifold $\cZ$, and show that there exists a flat metric on $\cZ$ that extends geometrically to each component of $\cM \mid \cZ$. This problem is easier if each component is totally orientable (thanks to our homological criterion, Proposition \ref{prop:leeb}), so we begin by considering a cover of $\cM$ with this property.  We say that a graph orbifold is \emph{totally orientable} if each Seifert component of its JSJ decomposition is totally orientable.

\begin{prop} \label{prop:orientation-cover}
Every graph orbifold $\cM$ has a totally-orientable cover $\cM^+\to\cM$ of degree $\le2$. 
\end{prop}

\begin{proof}
If $\cM$ is totally orientable, take $\cM^+=\cM$. Otherwise, fixing a JSJ decomposition $\cZ$ of $\cM$, we first specify the cover over each component of $\cM\mid\cZ$: for a totally-orientable component $\cX$ we take the disconnected double cover $\cX\sqcup\cX\to\cX$, and for a non-totally-orientable component $\cX$, we take its totally-orientable double cover $\cX^+\to\cX$. What remains is to explain why we can glue these appropriately to obtain a cover $\cM^+$ of $\cM$. For this we must check that if $\cX,\cX'$ are two components that meet along a component $\cT$ of $\cZ$, then the covers over $\cX$ and $\cX'$ agree over $\cT$. 

If $\cX$ is totally orientable, then each boundary component $\cT$ is a torus and the cover over this torus is the disconnected cover $\cT\sqcup\cT\to\cT$. If $\cX$ is not totally orientable, then its boundary components consist of tori and pillows. Over a torus, the cover is disconnected $\cT\sqcup\cT\to\cT$, and over a pillow, the cover is the hyperelliptic cover $T^2\to S^2(2,2,2,2)$ (cf.\ \S \ref{sec:TO-cover}). Thus, when two components of $\cM\mid\cZ$ meet along $\cT$, the two covers agree over $\cT$. We glue by an isomorphism of these covers to construct $\cM^+$. 
\end{proof}

\begin{rmk}
When $\cM$ is not totally orientable, the cover $p:\cM^+\to\cM$ constructed in the proof of Proposition \ref{prop:orientation-cover} has the additional property that over each totally orientable component of the JSJ decomposition of $\cM$, the cover is disconnected. 
\end{rmk}

\begin{rmk}
Since a double cover has two self-isomorphisms (as a covering space), the construction of $\cM^+$ is not unique.  
\end{rmk}

Set $\cM^+$ to be a totally-orientable double cover of $\cM$ if $\cM$ is not totally orientable, and set $\cM^+ = \cM$ otherwise. Define $N$ to be a connected component of the pullback $N'$ in the following diagram. 
\[\begin{tikzcd}
N' & {\cM^+} \\
M & {\mathcal{M}}
\arrow["{\text{}}", from=1-1, to=1-2]
\arrow[from=1-1, to=2-1]
\arrow[from=1-2, to=2-2]
\arrow[from=2-1, to=2-2]
\end{tikzcd}\]
Then $N\to M$ and $N\to\cM^+$ are covering maps, and the covering $N\to\cM$ is regular ($\pi_1(M)$ and $\pi_1(\cM^+)$ are both normal in $\pi_1(\cM)$, so their intersection, which is $\pi_1(N)$ is also normal). Write $H$ for the deck group of the cover $N\to\cM$.

The following Proposition \ref{prop:average-JSJ} constructs the desired flat metric on $\cZ$, using the regular manifold cover $N\to \cM$. 

\begin{prop}\label{prop:average-JSJ}
Fix $N$ and $\cM=N/H$ as above. Let $\mathcal Z\subset\mathcal M$ be a JSJ orbifold, and let $Z\subset N$ be the preimage of $\mathcal Z$ under the quotient map $N\to\mathcal M$. If $N$ has a nonpositively curved metric, then there is a flat metric $\mu_Z$ on $Z$ that is $H$-invariant and extends to a nonpositively curved metric on $N$. 
\end{prop}

\begin{proof}[Proof of Proposition \ref{prop:average-JSJ}] 
Our recipe for constructing $\mu_Z$ is as follows. Using a result of Leeb--Scott \cite{leeb-scott} we may find a flat metric on a JSJ orbifold $W$ of $N$ that extends to a nonpositively curved metric $\nu_W$ on $N$. Since the JSJ decomposition is unique, every component of $Z$ is isotopic to a unique component of $W$, and $\nu_W$ induces a flat metric $\nu_Z^{npc}$ (unique up to isotopy) on $Z$. Then by the homological criterion (Proposition \ref{prop:leeb}), the metric $\nu_Z^{npc}$  extends to a nonpositively curved metric on $N$. Separately, we may find an $H$-invariant flat metric $\nu_Z^{H}$ on $Z$ (by pulling back a flat metric from $Z/H$). The metric $\nu_Z^H$ has no reason to extend to a nonpositively curved metric on $N$. Nevertheless we use $\nu_Z^{H}$ to ``tune" $\nu_Z^{npc}$ by replacing $\nu_Z^{npc}$ with the unique flat metric $\mu_Z^0$ that's isotopic to $\nu_Z^{npc}$ and affinely equivalent to $\nu_Z^{H}$ on each component of $Z$ (Lemma \ref{lem:torus-pillow}), and we average $\mu_Z^0$ to obtain $\mu_Z$: 
\[
\mu_Z:=\frac{1}{|H|}\sum_{h\in H}h^*(\mu_Z^0).
\]
The metric $\mu_Z$ is flat because the metrics of $\mu_Z^0$ on components in the same $H$-orbit are affinely equivalent, and the average of affinely-equivalent flat metrics is flat (Remark \ref{rmk:average}). 

We show that $\mu_Z$ extends to a nonpositively curved metric on $N$ by applying the homological criterion to each component of $N\mid Z$. First note that $\mu_Z^0$ satisfies the homological criterion on each component because $\mu_Z^0$ is isotopic to $\nu_Z^{npc}$; the latter satisfies the homological criterion by assumption, and the homological criterion is clearly preserved by isotopy. Then $h^*(\mu_Z^0)$ also satisfies the homological criterion on each component because $h$ is a diffeomorphism. This implies that the convex combination $\mu_Z$ also satisfies the homological criterion on each component; see Remark \ref{rmk:convex-combination}.

When $X=T^2\times[0,1]$, in order to apply Proposition \ref{prop:leeb}, we need  the stronger condition that the metrics on the two boundary components are given by the same inner product $B$. This holds for $\nu^{npc}_Z$, so it also holds for the isotopic metric $\mu_Z^0$. Furthermore, it clearly holds for $h^*(\mu_Z^0)$ for each $h\in H$, so it also holds for the average $\mu_Z$. 
\end{proof}

To complete the proof of Theorem \ref{thm:main}, let $\mu_Z$ be an $H$-invariant flat metric on $Z$ that extends to a nonpositively curved metric on $N$ (from Proposition \ref{prop:average-JSJ}). The metric $\mu_Z$ descends to flat metrics $\mu_{\cZ^+},\mu_{\cZ}$ on $\cZ^+,\cZ$, respectively (where $\cZ^+$ is the preimage of $\cZ$ in $\cM^+$). We show that $\mu_{\cZ}$ extends geometrically to each component of $\cM \mid \cZ$. Fix a component $\cX$, and write $\cX^+,X$ for a component of the preimage of $\cX$ in $\cM^+,N$, respectively. Since $\mu_Z$ extends geometrically to $X$, Proposition \ref{prop:cover-totally-orientable} implies that $\mu_{\cZ^+}$ extends geometrically to $\cX^+$. Then Proposition \ref{prop:totally-orientable-cover} implies $\mu_{\cZ}$ extends geometrically to $\cX$. Since these extensions may be chosen to be flat on collar neighborhoods of $\partial\cX$, they glue together smoothly across all the Seifert components of $\cM \mid \cZ$ to produce a nonpositively curved metric on $\cM$.

\bibliographystyle{amsalpha}
\bibliography{references}

\end{document}